\documentclass{amsart}
\usepackage{amsmath,amssymb,amsthm,amsfonts}
\usepackage{paralist}
\usepackage{booktabs}
\usepackage[usenames,dvipsnames]{color}
\usepackage{comment}
\usepackage{graphicx}
\usepackage{latexsym}
\usepackage[latin1]{inputenc}
\usepackage[shortlabels]{enumitem}
\usepackage{thmtools}
\usepackage{url}
\usepackage[all,cmtip]{xy}
\usepackage{tikz}
\tikzset{font={\fontsize{4pt}{12}\selectfont}}
\usepackage{hyperref} 

\usepackage[enableskew]{youngtab}
\usepackage{ytableau}

\newcommand{\dss}{\displaystyle}

\newcommand{\Dyck}{\mathsf{Dyck}}
\newcommand{\nonpeaks}{\mathcal{NP}}
\newcommand{\FDyck}{\mathsf{FanDyck}}
\newcommand{\Alt}{\mathsf{Alt}}

\newcommand{\f}{{\bf f}}
\newcommand{\ssf}{\mathsf{f}}

\newcommand{\pleasant}{pleasant}

\title[Hook formulas for skew shapes II]{Hook formulas for skew
  shapes II. Combinatorial proofs \\ and enumerative applications}

\author[Alejandro Morales, Igor Pak, Greta Panova]{Alejandro H.~Morales$^\star$,
\ \ Igor Pak$^\star$, \ and \ \ Greta Panova$^\dagger$}

\thanks{\today}
\thanks{\thinspace ${\hspace{-.45ex}}^\star$Department of Mathematics,
UCLA, Los Angeles, CA~90095.
\hskip.06cm
Email:
\hskip.06cm
\texttt{\{ahmorales,\ts{pak}\}@math.ucla.edu}}

\thanks{\thinspace ${\hspace{-.45ex}}^\dagger$Department of Mathematics,
 UPenn, Philadelphia, PA~19104.
\hskip.06cm
Email:
\hskip.06cm
\texttt{panova@math.upenn.edu}}

\hypersetup{
  breaklinks, 
  pdftitle={Hook-length formulas for skew shapes II.}
}

\newcommand{\EP}{E^\ast}

\newcommand{\HP}{\mathcal{HP}}
\newcommand{\dd}{\mathsf{d}}
\newcommand{\PD}{\mathcal{P}}
\newcommand{\sPD}{p}
\newcommand{\ED}{\mathcal{E}}
\newcommand{\sED}{e}

\newcommand{\NIP}{\mathcal{NI}}

\newcommand{\BT}{\mathcal{BT}}
\newcommand{\PT}{\mathcal{PT}}

\newcommand{\RPP}{\operatorname{RPP}}
\newcommand{\SSYT}{\operatorname{SSYT}}
\newcommand{\SYT}{\operatorname{SYT}}

\DeclareMathOperator{\expeaks}{expk}

\DeclareMathOperator{\maj}{maj}

\DeclareMathOperator{\height}{ht}
\DeclareMathOperator{\outdeg}{deg}

\declaretheorem[numberwithin=section]{theorem}
\declaretheorem[numberlike=theorem]{lemma}
\declaretheorem[numberlike=theorem]{proposition}
\declaretheorem[numberlike=theorem]{corollary}
\declaretheorem[numberlike=theorem]{conjecture}

\declaretheorem[numberlike=theorem, style=definition]{definition}
\declaretheorem[numberlike=theorem, style=definition]{remark}

\declaretheorem[numberlike=theorem, style=definition]{example}

\numberwithin{equation}{section} 

\def\p{\mathsf{p}}
\def\wh{\widehat}

\def\emp{\varnothing}
\def\sq{\square}

\def\la{\lambda}
\def\ga{\gamma}
\def\si{\sigma}
\def\de{\delta}

\def\ssu{\subset}

\def\wt{\widetilde}
\def\<{\langle}
\def\>{\rangle}

\def\y{ {\text {\rm y}  } }

\def\0{{\mathbf 0}}

\def\SS{{S}}

\def\.{\hskip.06cm}
\def\ts{\hskip.03cm}

\def\nin{\noindent}

\begin{document}

\begin{abstract}
The \emph{Naruse hook-length formula} is a recent general formula for
the number of standard Young tableaux of skew shapes, given as a positive
sum over \emph{excited diagrams} of products of hook-lengths. In
\cite{MPP1} we gave two different $q$-analogues of Naruse's formula: for the
skew Schur functions, and for counting reverse plane partitions of skew
shapes.  In this paper we give an elementary proof of Naruse's formula
based on the case of border strips. For special border strips, we obtain curious
new formulas for the \emph{Euler} and \emph{$q$-Euler numbers} in
terms of certain Dyck path summations.
\end{abstract}


\ytableausetup{smalltableaux}

\maketitle

\section{Introduction} \label{sec:intro}

In Enumerative Combinatorics, when it comes to fundamental results,
one proof is rarely enough, and one is often on the prowl for a better,
more elegant or more direct proof.  In fact, there is a wide belief in
multitude of ``proofs from the Book'', rather than a singular best approach.
The reasons are both cultural and mathematical: different proofs
elucidate different aspects of the underlying combinatorial objects
and lead to different extensions and generalizations.

The story of this series of papers is on our effort to
understand and generalize the \emph{Naruse hook-length formula}
(NHLF) for the number of standard Young tableaux of a skew shape in
terms of \emph{excited diagrams}.
In our previous paper~\cite{MPP1}, we gave two
$q$-analogues of the NHLF, the first with an algebraic proof and
the second with a bijective proof.  We also gave a (difficult)
``mixed'' proof of the first $q$-analogue, which combined the
bijection with an algebraic argument.  Naturally, these provided
new proofs of the NHLF, but none which one would call ``elementary".

This paper is the second in the series.  Here we consider a special
case of border strips which turn out to be extremely fruitful both
as a technical tool and as an important object of study.
We give two elementary proofs of the NHLF in this case,
both inductive: one using weighted paths argument and another
using determinant calculation.  We then deduce the general case
of NHLF for all skew diagrams by using the
\emph{Lascoux--Pragacz identity} for Schur functions.
Since the latter has its own elementary proof~\cite{HGoul}
(see also~\cite{CYY}), we obtain an elementary proof of the~HLF.

But surprises do not stop here.  For the special cases of
the \emph{zigzag strips}, our approach gives a number of
curious new formulas for the \emph{Euler}
and two types of \emph{$q$-Euler numbers}, the second of
which seems to be new.  Because the
excited diagrams correspond to Dyck paths in this case,
the resulting summations have \emph{Catalan number} of terms.
We also give type~B analogues, which have a similar structure but
with $\binom{2n}{n}$ terms.  Despite their strong ``classical feel'', 
all these formulas are new and quite mysterious.

\medskip

\subsection{Hook formulas for straight and skew shapes}
Let us recall the main result from the first paper~\cite{MPP1} in this series.
We assume here the reader is familiar with the basic definitions, which
are postponed until the next two sections.

The {\em standard Young tableaux} (SYT)  of straight and skew shapes are
central objects in enumerative and algebraic combinatorics.  The number
$f^{\lambda} = |\SYT(\la)|$ of standard Young tableaux of shape $\lambda$ has the
celebrated {\em hook-length formula} (HLF):

\begin{theorem}[HLF; Frame--Robinson--Thrall~\cite{FRT}]
Let $\lambda$ be a partition of~$n$.  We have:
\begin{equation} \label{eq:hlf} \tag{HLF}
f^{\lambda} \, = \, \frac{n!}{\prod_{u\in [\lambda]} h(u)}\,,
\end{equation}
where $h(u)=\lambda_i-i+\lambda'_j-j+1$ is the {\em hook-length} of the
square $u=(i,j)$.
\end{theorem}

Most recently, Naruse generalized~\eqref{eq:hlf} as follows.
For a skew shape $\lambda/\mu$, an {\em excited diagram} is a subset of the
Young diagram~$[\lambda]$ of size $|\mu|$, obtained from the Young diagram~$[\mu]$
by a sequence of {\em excited moves}:
\begin{center}
\includegraphics{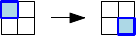}.
\end{center}
Such move $(i,j) \to (i+1,j+1)$ is allowed only if cells $(i,j+1)$,
$(i+1,j)$ and $(i+1,j+1)$ in $[\lambda]$ are unoccupied (see the precise definition and
an example in~$\S$\ref{ss:excited-def}).
We use $\ED(\lambda/\mu)$ to denote
the set of excited diagrams of~$\lambda/\mu$.

\begin{theorem}[NHLF; Naruse \cite{Strobl}] \label{thm:IN}
Let $\lambda,\mu$ be partitions, such that $\mu \ssu \la$.  We have:
\begin{equation} \label{eq:Naruse} \tag{NHLF}
f^{\lambda/\mu} \,  = \, |\la/\mu|! \, \sum_{D \in \ED(\lambda/\mu)}\,\.\.
 \prod_{u \in [\lambda]\setminus D} \frac{1}{ h(u)}\.\ts\..
\end{equation}
\end{theorem}

When $\mu = \emp$, there is a unique excited diagram $D=\emp$, and we obtain
the usual~HLF.

\smallskip

The NHLF has two natural $q$-analogues which were proved in the previous
paper in the series.

\begin{theorem}[\cite{MPP1}] \label{thm:skewSSYT}
We have:
\begin{equation} \label{eq:skewschur}  \tag{first $q$-NHLF}
s_{\lambda/\mu}(1,q,q^2,\ldots) \, = \, \sum_{S\in \ED(\lambda/\mu)}
\. \.\.\prod_{(i,j) \in [\lambda]\setminus S}\frac{q^{\lambda'_j-i}}{1-q^{h(i,j)}}\ts.
\end{equation}
\end{theorem}

\begin{theorem}[\cite{MPP1}] \label{thm:skewRPP}
We have:
\begin{equation} \label{eq:skew-RPP} \tag{second $q$-NHLF}
\sum_{\pi \in \RPP(\lambda/\mu)} q^{|\pi|} \, = \, \sum_{S \in
 \PD(\lambda/\mu)} \. \.\prod_{u\in S}\frac{q^{h(u)}}{1-q^{h(u)}}\,,
\end{equation}
where $\PD(\lambda/\mu)$ is the set of pleasant diagrams $($see {\rm Definition~\ref{def:agog}}$)$.
\end{theorem}

The second theorem employs a new family of combinatorial objects called
{\em pleasant diagrams}.  These diagrams can be defined as
subsets of complements of excited diagrams,
and are technically useful.  This allows us to write
the RHS of~\eqref{eq:skew-RPP} completely in terms of excited diagrams (see \cite[\S 6]{MPP1}).

\smallskip

\subsection{Combinatorial proofs}

Our approach to the combinatorial proof of the NHLF in
Section~\ref{sec:all_shapes}-\ref{sec:2ndproof} is as follows. We
start by proving the case of border strips (connected skew shapes with no $2\times
2$ square). In this case the NHLF is more elegant,
\[
\frac{f^{\lambda/\mu}}{|\lambda/\mu|!} = \sum_{\ga}
\prod_{(i,j) \in \ga} \frac{1}{h(i,j)},
\]
where the sum is over lattice paths $\ga$ from $(\lambda'_1,1)$ to
$(1,\lambda_1)$ that stay inside $[\lambda]$. We give two self
contained inductive proofs of
this case. The first proof in Section~\ref{sec:1stproof} is based on
showing a multivariate identity of paths. The second proof in
Section~\ref{sec:2ndproof} uses determinants to show that a multivariate
identity of paths equals a ratio of \emph{factorial Schur functions}.

We then use a corollary of the Lascoux--Pragacz identity for skew Schur
functions: if $(\theta_1,\ldots,\theta_k)$ is a decomposition of the
shape $\lambda/\mu$ into outer border strips $\theta_i$ (see
Section~\ref{subsec:LP}) then
\[
\frac{f^{\lambda/\mu}}{|\lambda/\mu|!} \,=\,  \det\left[\,  \frac{f^{\theta_i
      \# \theta_j}}{|\theta_i \# \theta_j|!} \,\right]_{i,j=1}^k,
\]
where $\theta_i\#\theta_j$ is a certain substrip of the outer border
strip of $\lambda$.

Combining the case for border strips and this determinantal identity
we get
\[
\frac{f^{\lambda/\mu}}{|\lambda/\mu|!} \,=\, \det\left[\,
  \sum_{\substack{\ga:(a_j,b_j) \to (c_i,d_i),\\\ga\subseteq \lambda}}
  \prod_{(r,s)\in \ga} \frac{1}{h(r,s)} \,\right]_{i,j=1}^k,
\]
where $(a_j,b_j)$ and $(c_i,d_i)$ are the endpoints of the border
strip $\theta_i\# \theta_j$. Lastly, using the Lindstr\"om--Gessel--Viennot lemma this
determinant is written as a weighted sum over non-intersecting
lattice paths in $\lambda$. By an explicit characterization of excited
diagrams in Section~\ref{sec:excited_diagrams}, the supports of such paths are exactly the complements of
excited diagrams. The NHLF then follows.

A similar approach is used in Section~\ref{sec:all_shapes} to give a
combinatorial proof of the first $q$-NHLF for all skew shapes given in~\cite{MPP1}.
The Hillman--Grassl inspired bijection in~\cite{MPP1} remains the only combinatorial
proof of the second $q$-NHLF.

\begin{remark}
We should mention that our inductive proof is involutive, but basic enough
to allow ``bijectification'', i.e.~an involution principle proof of the NHLF.
We refer to \cite{Kratt-invol,Rem,Zei} for the involution principle proofs
of the (usual) HLF.
\end{remark}

\smallskip

\subsection{Enumerative applications}\label{ss:intro-enum}
In sections~\ref{sec:enum_strips_SSYT} and \ref{sec:enum_strips_RPP},
we give enumerative formulas which follow from NHLF.  They involve $q$-analogues
of Catalan, Euler and Schr\"oder numbers. We highlight several of these formulas.

Let $\Alt(n)= \{\si(1)<\si(2)>\si(3)<\si(4)>\ldots\} \ssu \SS_n$ be the set of
{\em alternating permutations}.
The number $E_n=|\Alt(n)|$ is the $n$-th {\em Euler number} (see~\cite{Stanley_SurveyAP}
and \cite[\href{http://oeis.org/A000111}{A000111}]{OEIS}), with the
generating function
\begin{equation}\label{eq:tan-sec}
\sum_{n=0}^\infty \. E_n \. \frac{x^n}{n!} \,\. = \,\, \tan(x) \ts + \ts \sec(x)\..
\end{equation}

Let $\delta_n =(n-1,n-2,\ldots,2,1)$ denote the staircase shape and observe that
$E_{2n+1}= f^{\de_{n+2}/\de_n}$.  Thus, the NHLF relates Euler numbers with excited
diagrams of $\de_{n+2}/\de_n$.  It turns out that these excited diagrams are
in correspondence with the set $\Dyck(n)$ of {\em Dyck paths} of length~$2n$
(see Corollary~\ref{cor:excitedcat}).  More precisely,
\[
|\ED(\de_{n+2}/\de_n)| \, = \, |\Dyck(n)| \, = \, C_n \, = \, \frac{1}{n+1}\binom{2n}{n}\ts,
\]
where $C_n$ is the $n$-th Catalan number, and $\Dyck(n)$ is the set of lattice paths from
$(0,0)$ to $(2n,0)$ with steps $(1,1)$ and $(1,-1)$ that stay on or above
the $x$-axis  (see e.g. \cite{StCat}). Now the NHLF implies the following identity.
\begin{corollary}\label{cor:euler-nhlf}
We have:
\begin{equation}\label{eq:cat2euler} \tag{EC}
\sum_{\p \in \Dyck(n)} \, \prod_{(a,b) \in   \p}\frac{1}{2b+1}
\, = \,  \frac{E_{2n+1}}{(2n+1)!} \ts\,,
\end{equation}
where $(a,b) \in \p$ denotes a point $(a,b)$ of the Dyck
path~$\p$.
\end{corollary}

Consider the following two $q$-analogues of $E_n$, the first of which
is standard in the literature:

$$
E_n(q):=\sum_{\sigma\in
    \Alt(n)} q^{\maj(\sigma^{-1})} \qquad \text{and} \qquad
\EP_n(q) :=\sum_{\sigma\in
    \Alt(n)} q^{\maj(\sigma^{-1} \kappa)}\,,
$$
where $\maj(\sigma)$ is the {\em major index} of permutation $\sigma$
in $\SS_n$ and $\kappa$ is the permutation
$\kappa=(13254\ldots)$.
See examples~\ref{ex:qEulerA} and~\ref{ex:qEulerB} for the initial values.

Now, for the skew shape $\de_{n+2}/\de_n$, Theorem~\ref{thm:skewSSYT}
gives the following $q$-analogue of Corollary~\ref{cor:euler-nhlf}.
\begin{corollary}\label{cor:euler-nhlf-ssyt}
We have:
\[
\sum_{\p \in \Dyck(n)} \. \prod_{(a,b) \in
  \p}\frac{q^b}{1-q^{2b+1}}
\, = \,
\frac{E_{2n+1}(q)}{(1-q)(1-q^2)\cdots (1-q^{2n+1}) }
\ts\..
\]
\end{corollary}

Similarly, Theorem~\ref{thm:skewRPP} in this case
gives a different $q$-analogue.
\begin{corollary}\label{cor:euler-nhlf-rpp}
We have:
$$
\sum_{\p \in \Dyck(n)} \,
  q^{H(\p)} \, \.\prod_{(a,b)\in \p} \frac{1}{1-q^{2b+1}}
\, = \,
\frac{\EP_{2n+1}(q)}{(1-q)(1-q^2)\cdots (1-q^{2n+1})}\ts\.,
$$
where
$$H(\p) \, = \, \sum_{(c,d)\in \HP(\p)} \, (2d+1) \.,
$$
and \. $\HP(\p)$ \. denotes the set of peaks $(c,d)$ in $\p$ with height $d>1$.
\end{corollary}

All three corollaries are derived in sections~\ref{sec:enum_strips_SSYT}
and~\ref{sec:enum_strips_RPP}.

\smallskip

Section~\ref{sec:enum_strips_SSYT} considers the special case
when $\lambda/\mu$ is a thick strip shape $\delta_{n+2k}/\delta_n$,
which gives the connection with Euler and Catalan numbers.
In Section~\ref{sec:enum_strips_RPP},
we consider the pleasant diagrams of the thick strip shapes,
establishing a
connection with Schr\"oder numbers. We also state conjectures
on certain determinantal formulas.  We conclude with final remarks and open problems in Section~\ref{sec:finrem}.

\bigskip\section{Notation and Background} \label{sec:notation}

\subsection{Young diagrams} \label{ss:not-yd}
Let $\lambda=(\lambda_1,\ldots,\lambda_r),
\mu=(\mu_1,\ldots,\mu_s)$ denote integer partitions of
length $\ell(\lambda)=r$ and $\ell(\mu)=s$. The {\em size} of the partition
is denoted by $|\lambda|$ and $\lambda'$
denotes the {\em conjugate partition} of $\lambda$. We use $[\lambda]$ to
denote the Young diagram of the partition $\lambda$. The \emph{hook length}
$h_{ij} = \la_i - i +\la_j' -j +1$ of a square $u=(i,j)\in [\la]$ is
the number of squares directly to the right and directly below~$u$
in~$[\la]$ including $u$.  The {\em Durfee square} $\square^{\lambda}$ is the largest
square of the form $\{(i,j), 1\le i,j \le k\}$ inside~$[\la]$. Let $\square_k^\lambda$ be the largest $i
\times (i+k)$ rectangle that fits inside the Young diagram starting at
$(1,1)$. For $k=0$, the rectangle $\square_0^{\lambda}=\sq^\la$ is the Durfee
square of~$\lambda$.

A {\em skew shape} is denoted by $\lambda/\mu$. A skew shape can have
multiple edge connected components.
 For an integer $k$, let $\dd_k$ be the diagonal
$\{(i,j) \in \lambda/\mu \mid i-j = k\}$,
where $\mu_k=0$ if $k> \ell(\mu)$. For an integer $t$, $1\leq t
\leq \ell(\lambda)-1$ let $\dd_t(\mu)$ denote the diagonal $\dd_{\mu_t-t}$
where $\mu_t=0$ if $\ell(\mu)<t \leq \ell(\lambda)$.

Given the skew shape $\lambda/\mu$, let $P_{\lambda/\mu}$ be the poset
of cells $(i,j)$ of $[\lambda/\mu]$ partially ordered by component. This poset is {\em
  naturally labelled}, unless otherwise stated.


\subsection{Young tableaux} \label{ss:not-yt}
A {\em reverse plane partition} of skew shape
$\lambda/\mu$ is an array $\pi=(\pi_{ij})$ of nonnegative integers of shape $\lambda/\mu$
that is weakly increasing in rows and columns. We denote the set of
such plane partitions by $\RPP(\lambda/\mu)$.
 A {\em semistandard Young tableau} of shape $\lambda/\mu$ is a RPP
of shape $\lambda/\mu$ that is
strictly increasing in columns. We denote the set of such tableaux by
$\SSYT(\lambda/\mu)$. A {\em standard Young
  tableau} (SYT) of shape $\lambda/\mu$ is an array $T$ of shape
$\lambda/\mu$ with the numbers $1,\ldots,n$, where
$n=|\lambda/\mu|$, each $i$
appearing once, strictly increasing in rows and columns.  For example,
there are five SYT of shape $(32/1)$:
\[
\ytableausetup{smalltableaux}\ytableaushort{\none12,34}\ \quad
\ytableaushort{\none13,24}\ \quad  \ytableaushort{\none14,23}\ \quad
\ytableaushort{\none23,14}\ \quad \ytableaushort{\none24,13}
\]
 The {\em
  size} of a RPP or tableau $T$ is the
sum of its entries and is denoted by $|T|$.

\subsection{Skew Schur functions} \label{ss:not-sym}
Let $s_{\lambda/\mu}({\bf x})$ denote the  {\em skew Schur function}
of shape $\lambda/\mu$  in variables
${\bf x} = (x_0,x_1,x_2,\ldots)$. In particular,
\[
s_{\lambda/\mu}({\bf x}) \. = \. \sum_{T\in \SSYT(\lambda/\mu)} {\bf x}^T\., \ \  \qquad
s_{\lambda/\mu}(1,q,q^2,\ldots) \. = \. \sum_{T\in \SSYT(\lambda/\mu)} q^{|T|}
\.,
\]
where \ts ${\bf x}^T = x_0^{\#0s \text{ in } (T)}\ts x_1^{\#1s \text{
    in } (T)}\ldots$ \ts
Recall, the skew shape $\lambda/\mu$ can have multiple edgewise-connected
components $\theta_1,\ldots,\theta_m$. Since then $s_{\lambda/\mu} =
s_{\theta_1} \cdots s_{\theta_k}$ we assume without loss of generality
that $\lambda/\mu$ is edgewise connected.

\subsection{Determinantal identities for $s_{\lambda/\mu}$}

The Jacobi-Trudi identity (see e.g.~\cite[\S 7.16]{EC2}) states that
\begin{equation} \label{eq:JTid}
s_{\lambda/\mu}({\bf x}) \, = \, \det\bigl[h_{\lambda_i-\mu_j -i+j}({\bf x})\bigr]_{i,j=1}^{n},
\end{equation}
where $h_k({\bf x}) = \sum_{i_1\leq i_2 \leq \cdots \leq i_k}
x_{i_1}x_{i_2}\cdots x_{i_k}$ is the $k$-th {\em complete symmetric
  function}.

There are other determinantal identities of (skew) Schur functions
like the Giambelli formula (e.g. see \cite[Ex. 7.39]{EC2}) and the {\em Lascoux--Pragacz}
identity \cite{LP}. Hamel and Goulden \cite{HGoul} found a vast common generalization to these
three identities by giving an exponential number of determinantal
identities for $s_{\lambda/\mu}$ depending on {\em outer
  decompositions} of the shape $\lambda/\mu$. We focus on the
Lascoux--Pragacz identity that we describe next through the
Hamel--Goulden theory (e.g. see \cite{CYY}).

A {\em border strip} is a connected skew shape without any $2\times 2$
squares. The starting point and ending point of a strip are its
southwest and northeast endpoints. Given $\lambda$, the {\em outer border strip}  is the strip
containing all the boxes inside $[\lambda]$ sharing a vertex with the boundary of
$\lambda$., i.e. $\lambda/(\lambda_2-1,\lambda_3-1,\ldots)$. A {\em Lascoux--Pragacz}
decomposition of $\lambda/\mu$ is a decomposition of the skew shape
into $k$ maximal outer border
strips $(\theta_1,\ldots,\theta_k)$, where $\theta_1$ is the outer
border strip of $\lambda$, $\theta_2$ is the outer border strip of the
remaining diagram $\lambda\setminus \theta_1$, and so on until we
start intersecting $\mu$. In this case, we continue the decomposition
with each remaining connected component. The strips are ordered
$\succeq$ by the contents of their northeast endpoints. See
Figure~\ref{exLascouxPragacz}:Left, for an example.

We call the border strip $\theta_1$ the {\em cutting strip} of the
decomposition and denote it by $\tau$ \cite{CYY}. For integers $p$ and
$q$, let $\phi[p,q]$ be the substrip of $\tau$ consisting of the cells
with contents between $p$ and $q$. By convention,
$\phi[p,p]=(1)$, $\phi[p+1,p]=\varnothing$ and $\phi[p,q]$ with
$p>q+1$ is undefined.
The strip
$\theta_i \# \theta_j$ is the substrip
$\phi[p(\theta_j),q(\theta_i)]$ of $\tau$,  where $p(\theta_i)$ and
$q(\theta_i)$ are the contents of the starting point and ending point of
$\theta_i$.

\begin{theorem}[Lascoux--Pragacz \cite{LP}, Hamel--Goulden \cite{HGoul}] \label{thm:LascouxPragacz}
If $(\theta_1,\ldots,\theta_k)$ is a Lascoux--Pragacz decomposition of
$\lambda/\mu$, then
\begin{equation} \label{eq:LPid}
s_{\lambda/\mu} = \det{\big [} \, s_{\theta_i \# \theta_j}\, {\big ]}_{i,j=1}^k.
\end{equation}
where $s_{\varnothing}=1$ and $s_{\phi[p,q]} =0$ if $\phi[p,q]$ is undefined.
\end{theorem}

\begin{example} \label{ex:LP}
Figure~\ref{exLascouxPragacz}:Right, shows the Lascoux--Pragacz decomposition for the
shape $\lambda/\mu=(5441/21)$ into two strips
$(\theta_1,\theta_2)$ where $\tau=\theta_1 = (5441/33)$ and $\theta_2
= (33/21)$. The substrips of $\tau$ appearing in the matrix of the identity are
\[
\theta_1 \# \theta_1 = \theta_1, \quad \theta_1 \# \theta_2 =
\phi[0,4] = (322/11), \quad \theta_2 \# \theta_1 = \phi[-3,2] =
(441/3), \quad \theta_2\# \theta_2 = \theta_2.
\]
Then by the Lascoux--Pragacz identity  $s_{(5441/21)}$ can be written
as the following  $2\times 2$
determinant
\[
s_{\lambda/\mu} = \det \begin{bmatrix} s_{\theta_1 \# \theta_1} &
  s_{\theta_1 \# \theta_2} \\ s_{\theta_2\# \theta_1} & s_{\theta_2 \#
    \theta_2} \end{bmatrix} = \det \begin{bmatrix} s_{5441/33} & s_{322/11} \\
  s_{441/3} & s_{22/1}\\ \end{bmatrix}.
\]
\end{example}

\begin{figure} 
\includegraphics{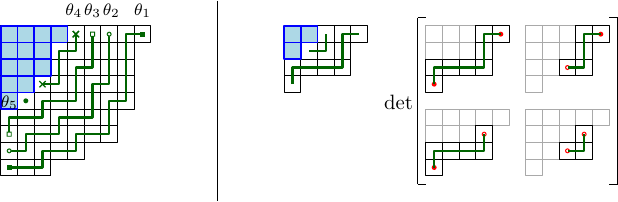}
\caption{Left: Example of a
  Lascoux--Pragacz outer decomposition. Right: the
  Lascoux--Pragacz decomposition of the shape $\lambda/\mu=(5441/21)$ and the border strips of
  the determinantal identity for $s_{\lambda/\mu}$}
\label{exLascouxPragacz}
\end{figure}

\subsection{Factorial Schur functions} \label{sec:factorial_schurs}

The
{\em factorial Schur function} (e.g. see \cite{MS}) is defined as
\[
s_{\mu}^{(d)}({\bf x} \mid {\bf a}) := \frac{\det \left[
    (x_i-a_1)\cdots (x_i -
    a_{\mu_j+d-j})\right]_{i,j=1}^d}{\dss \prod_{1\leq i < j \leq d}(x_i-x_j)},
\]
where ${\bf x} = x_1,\ldots,x_d$ and ${\bf a} =a_1,a_2,\ldots$ is a
sequence of parameters.

\subsection{Permutations}  We write permutations of $\{1,2,\ldots,n\}$
in {\em one-line notation}: $w=(w_1w_2\ldots w_n)$ where $w_i$ is the
image of $i$. A {\em descent} of $w$ is an index $i$ such that
$w_i>w_{i+1}$. The {\em major index} $\maj(w)$ is the sum $\sum i$ of
all the descents $i$ of $w$.

\subsection{Dyck paths} \label{ss:not-grid}
A {\em Dyck path}~$\p$ of length $2n$ is a lattice paths from
$(0,0)$ to $(2n,0)$ with steps $(1,1)$ and $(1,-1)$ that stay on or above
the $x$-axis. We use $\Dyck(n)$ to denote the set of Dyck paths of length~$2n$.
For a Dyck path $\p$, a {\em peak} is a point $(c,d)$ such
that $(c-1,d-1)$ and $(c+1,d-1)\in \p$.  Peak $(c,d)$ is called
a {\em high-peak} if~$d>1$.

\bigskip\section{Excited diagrams}\label{sec:excited_diagrams}

\subsection{Definition} \label{ss:excited-def}

Let $\lambda/\mu$ be a skew partition and $D$ be a subset of the Young
diagram of $\lambda$. A cell $u=(i,j) \in D$ is called {\em active} if
  $(i+1,j)$, $(i,j+1)$ and $(i+1,j+1)$ are all in
$[\lambda]\setminus D$.  Let $u$ be an
active cell of $D$, define $\alpha_u(D)$ to be the set obtained by
replacing $(i,j)$ in $D$ by $(i+1,j+1)$. We call this replacement an {\em excited move}. An {\em excited diagram} of
$\lambda/\mu$ is a subdiagram of $\lambda$ obtained from the Young
diagram of $\mu$
after a sequence of excited moves on active cells. Let
$\ED(\lambda/\mu)$ be the set of
excited diagrams of~$\lambda/\mu$ and $\sED(\lambda/\mu)$ its
cardinality. For example, Figure~\ref{fig:excited} shows the eight
excited diagrams of $(5441/21)$ (for the moment ignore the paths in
the complement).

\begin{figure}[hbt]
\begin{center}
\includegraphics{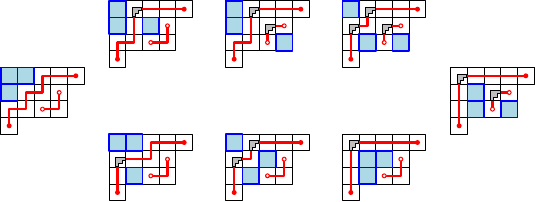}
\caption{The eight excited diagrams (in blue) of shape
  $\lambda/\mu=(5441/21)$ and the corresponding non-intersecting paths
  (in red) in their complement. The high peaks of the paths described
  in Section~\ref{subsec:pleasant} are in gray.}
\label{fig:excited}
\end{center}
\end{figure}


\medskip \subsection{Flagged tableaux}\label{ss:excited-flagged}
Excited diagrams of $\lambda/\mu$ are
in bijection with certain {\em flagged tableaux} of shape~$\mu$
(see \cite[\S 3]{MPP1} and \cite[\S 6]{VK}). Thus, the number of
excited diagrams is given by a determinant, a
polynomial in the parts of $\lambda$ and~$\mu$ as follows. 

Consider the diagonal that passes through cell $(i,\mu_i)$, i.e. the
last  cell of row $i$ in $\mu$. Let this diagonal intersect the
boundary of $\lambda$ at a row denoted by
$\ssf^{(\lambda/\mu)}_i$. Let $\mathcal{F}(\lambda/\mu)$ be the set of
SSYT of shape $\mu$ with positive entries in row $i$ that are at most
$\ssf^{(\lambda/\mu)}_i$. Such SSYT are called flagged tableaux and
$(\ssf^{\lambda/\mu}_1,\ldots,\ssf^{\lambda/\mu}_{\ell})$ is called a \emph{flag}. Given an excited diagram $D$ in
$\ED(\lambda/\mu)$ each
$(x,y)$ in $[\mu]$ corresponds to a cell $(i_x,j_y)$ in $D$, let
$\varphi(D):=T$ be the tableau of shape $\mu$ with $T_{x,y} =
j_y$. See Figure~\ref{fig:excited2cat} for an example of $\varphi$.

\begin{proposition}[\S 3 \cite{MPP1}]  \label{prop:excited2flagged}
We have that $\sED(\lambda/\mu) = |\mathcal{F}(\lambda/\mu)|$ and
$\varphi$ is a bijection between these two sets.
\end{proposition}

By the enumeration of flagged
tableaux of Wachs \cite{W85} the next formula follows.

\begin{proposition}[\S 3 \cite{MPP1}] \label{prop:GV}
In notation above, we have:
\[
\sED(\lambda/\mu) = \det\left[
  \binom{\ssf^{(\lambda/\mu)}_i+\mu_i-i+j-1}{\ssf^{(\lambda/\mu)}_i-1}\right]_{i,j=1}^{\ell(\mu)}.
\]
\end{proposition}

\begin{example}
For the same shape as in Example~\ref{ex:LP} and Figure~\ref{fig:excited}, the number of excited
diagrams equals the number of flagged tableaux of shape $(2,1)$ with
entries in the first and second row $\leq 3$. Thus
\[
\sED(\lambda/\mu) = \det\begin{bmatrix} \binom{4}{2} & \binom{5}{2} \\[2pt]
  \binom{2}{2} & \binom{3}{2} \end{bmatrix}= \det \begin{bmatrix} 6
  &10 \\ 1 & 3\end{bmatrix} = 8.
\]
\end{example}


\subsection{Border strip decomposition formula for $\sED(\lambda/\mu)$} \label{sec:excited_paths}

This first determinantal identity for $\sED(\lambda/\mu)$ is similar
to the Jacobi--Trudi identity for $s_{\mu}$. In this section we prove
a new determinantal identity for
$\sED(\lambda/\mu)$ very similar to the Lascoux--Pragacz identity for $s_{\lambda/\mu}$.

\begin{theorem} \label{thm:num_excited_HG}
If $(\theta_1,\ldots,\theta_k)$ is the Lascoux--Pragacz decomposition of $\lambda/\mu$
into $k$ maximal outer border strips then
\[
\sED(\lambda/\mu) = \det {\big [}\, \sED({\theta_i \# \theta_j})\, {\big ]}_{i,j=1}^k,
\]
where $\sED({\varnothing})=1$ and $\sED({\phi[p,q]}) =0$ if $\phi[p,q]$ is undefined.
\end{theorem}

\begin{example} \label{ex:Kpaths}
For the same shape $\lambda/\mu$ as in Example~\ref{ex:LP} and
Figure~\ref{fig:excited} we have
\[
\sED({\lambda/\mu}) = \det \begin{bmatrix} \sED({5441/33}) & \sED({322/11}) \\
  \sED({441/3}) & \sED({22/1})\\ \end{bmatrix} =\det\begin{bmatrix} 10 & 3
  \\ 4  & 2 \end{bmatrix} =8.
\]
\end{example}

In order to prove Theorem~\ref{thm:num_excited_HG} we show a relation
between excited diagrams and certain tuples of
non-intersecting paths in the diagram of $\lambda$. The support of a
path are the cells of $[\lambda]$ contained in the path.

For the connected skew shape $\lambda/\mu$ there is a unique tuple of
border-strips (i.e. non-intersecting paths)
 $(\ga^*_1,\ldots,\ga^*_k)$ in $\lambda$ with support
$[\lambda/\mu]$, where each border strip $\ga^*_i$ begins at the southern box
$(a'_i,b'_i)$ of a column and ends at the eastern box $(c'_i,d'_i)$ of
a row
\cite[Lemma 5.3]{VK}. We call this tuple the \emph{Kreiman
  decomposition} of $\lambda/\mu$. Let $\NIP(\lambda/\mu)$ be the set of $k$-tuples
$\Gamma:=(\ga_1,\ldots,\ga_k)$ of non-intersecting paths contained in
$[\lambda]$ with $\ga_i:(a_i,b_i)\to (c_i,d_i)$. Kreiman showed that the supports of the
paths in $\NIP(\lambda/\mu)$ are exactly the complements of excited
diagrams in $\ED(\lambda/\mu)$ \cite[$\S$5, $\S$6]{VK}. See
Figure~\ref{fig:excited}, for an example. We include a
proof of this result. The first half of the argument uses a lemma from
Kreiman, the rest of the argument is a new proof using pleasant diagrams
(see Section~\ref{sec:pleasant}) instead of the induction on the number of excited moves.

\begin{proposition}[Kreiman~\cite{VK}] \label{prop:NIP2ED}
The $k$-tuples of paths in $\NIP(\lambda/\mu)$ are uniquely determined by
their supports in $[\lambda]$  and moreover these supports are exactly the complements of
excited diagrams of~$\lambda/\mu$.
\end{proposition}

\begin{proof}
The fact that paths are uniquely determined by their support in
$[\lambda]$ follows by \cite[Lemma 5.2]{VK}. By abuse of notation we
identify the $k$-tuples of paths in $\NIP(\lambda/\mu)$ with their
supports. Note that the supports of all $k$-tuples of paths in
$\NIP(\lambda/\mu)$ have size $|\lambda/\mu|$.

We now show that the support of $k$-tuples in $\NIP(\lambda/\mu)$
correspond to complements of excited diagrams. First, we show that if $D\in \ED(\lambda/\mu)$ then $[\lambda]\setminus
D \in \NIP(\lambda/\mu)$ by induction on the number of excited moves. Given $[\mu] \in \ED(\lambda/\mu)$, its complement $[\lambda/\mu]$ corresponds to Kreiman decomposition
$(\ga^*_1,\ldots,\ga^*_k) \in \NIP(\lambda/\mu)$ as mentioned above. Then excited moves on
the diagrams correspond to {\em ladder moves} on the non-intersecting
paths:
\begin{center}
\includegraphics{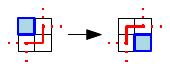}.
\end{center}
The latter do not introduce intersections and preserve
the endpoints of the paths. Thus for each $D \in \ED(\lambda/\mu)$,
its complement $[\lambda]\setminus D$ corresponds to a $k$-tuple
$(\ga_1,\ldots,\ga_k)$ of paths in $\NIP(\lambda/\mu)$.

Conversely, consider the support $S$ of a $k$-tuple of paths in
$\NIP(\lambda/\mu)$. The set $S$ has the following property: the subset $S_k := S \cap
\square^{\lambda}_k$ has no descending chain bigger than the length of
the $k$th diagonal of $\lambda/\mu$. Such sets $S$ are called {\em
  pleasant diagram} of
$\lambda/\mu$ \cite[\S$6$]{MPP1} (see Section~\ref{sec:pleasant}). Since $|S|=|\lambda/\mu|$, by
\cite[Thm. 6.5]{MPP1} the set $S$ is the complement of an excited
diagram in $\ED(\lambda/\mu)$, as desired.
\end{proof}

\begin{corollary}  \label{cor:excited_eqs_NIP} $\sED(\lambda/\mu) \, = \, |\NIP(\lambda/\mu)|$.
\end{corollary}

Next, we show a correspondence between the Kreiman decomposition and
the Lascoux--Pragacz decomposition of $\lambda/\mu$.

\begin{lemma} \label{lem:contentsLPvsK}
There is a correspondence between the border strips of the Lascoux--Pragacz
and the Kreiman decomposition of $\lambda/\mu$. This correspondence preserves the
lengths and contents of the starting and ending
points of the paths/border strips:
\[
p(\theta_i) = b_i-a_i, \quad q(\theta_i) = d_i-c_i.
\]
\end{lemma}

\begin{proof}
Let $(\theta_1,\ldots,\theta_k)$ and $(\ga^*_1,\ldots,\ga^*_t)$ be the
Lascoux--Pragacz and the Kreiman decomposition of the shape
$\lambda/\mu$. We prove the result by induction on $k$. Note that $\theta_1$ and $\ga^*_1$ have the same
endpoints $(\lambda'_1,1)$ and $(1,\lambda_1)$. Thus the strips have
the same length and their endpoints have the same respective contents.

Next, note that the skew shapes $\lambda/\mu$ with $\theta_1$ removed
and $\lambda/\mu$ with $\ga^*_1$ removed are the same up to
shifting. The Lascoux--Pragacz decomposition
of this new shape is $(\theta_2,\ldots,\theta_k)$ with the contents of
the endpoints unchanged. Similarly, the Kreiman decomposition of this
new shape is $(\ga^*_2,\ldots,\ga^*_t)$ with the contents of the endpoints
unchanged. By induction $k-1=t-1$ and the strip $\theta_i$ and
$\ga^*_i$ for $i=2,\ldots,k$ have the same length and their endpoints
have the same content. This completes the proof.
\end{proof}

\begin{figure}
\label{fig:contents_paths}
\includegraphics{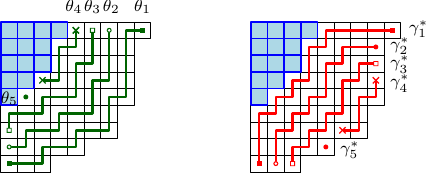}
\caption{Example of the correspondence of the border strips of the
  Lascoux--Pragacz outer decomposition $(\theta_1,\ldots,\theta_5)$ and
  Kreiman outer decomposition $(\ga^*_1,\ldots,\ga^*_5)$ of
  $\lambda/\mu$. The contents of the endpoints and the lengths of
  $\theta_i$ and $\ga^*_i$ are the same.}
\end{figure}

Lastly, we need a Lindstr\"om--Gessel--Viennot type lemma to count (weighted)
non-intersecting paths in $[\lambda]$. To state the Lemma we need some
notation. Its proof follows the usual sign-reversing involution on
paths that intersect (e.g. see \cite[\S 2.7]{EC2}). Let $(a_1,b_1),\ldots,
(a_k,b_k)$ and $(c_1,d_1),\ldots,(c_k,d_k)$ be cells in
$[\lambda]$ and let
\[
h((a_j,b_j) \to (c_i,d_i), {\bf y}) := \sum_{\ga} \prod_{(r,s) \in \ga} y_{r,s},
\]
where the sum is over paths $\ga:(a_j,b_j)\to (c_i,d_i)$ in
$[\lambda]$ with steps $(1,0)$ and $(0,1)$, and the product is over
cells $(r,s)$ of $\ga$. Let also
\[
N_{\lambda}((a_i,b_i)\to (c_i,d_i); {\bf y}) := \sum_{(\ga_1,\ldots,\ga_k)} \prod_{i=1}^k \prod_{(r,s)\in \ga_i} y_{r,s},
\]
where the sum is over $k$-tuples $(\ga_1,\ldots,\ga_k)$ of
non-intersecting paths $\ga_i:(a_i,b_i)\to (c_i,d_i)$ in
$[\lambda]$.

\begin{lemma}[Lindstr\"om--Gessel--Viennot]  \label{lem:LGVlambda}
\[
N_{\lambda}((a_i,b_i)\to (c_i,d_i), {\bf y})  = \det{\bigr[} h((a_j,b_j)\to
  (c_i,d_i), {\bf y}) {\bigl]}_{i,j=1}^k.
\]
\end{lemma}

\begin{proof}[Proof of Theorem~\ref{thm:num_excited_HG}]
Combining Corollary~\ref{cor:excited_eqs_NIP}  and
Lemma~\ref{lem:LGVlambda} with weights $y_{r,s}=1$ it
follows that $\sED(\lambda/\mu)$ can be calculated by the determinant:

\begin{equation}
\label{eq:excited_2nddet}
\sED(\lambda/\mu) = \det {\big [} \#\{\text{ paths } \ga \mid  \ga
\subseteq \lambda, \ga:(a_j,b_j) \to (c_i,d_i)\,\}  {\big ]}_{i,j=1}^k.
\end{equation}

Now, by Corollary~\ref{cor:excited_eqs_NIP}, the number of paths
in  each matrix entry in the RHS of \eqref{eq:excited_2nddet} is also
the number of excited diagrams of a border strip $\theta$. Since in Kreiman's
decomposition, the endpoints $(a_j,b_j)$ and $(c_i,d_i)$ are in the
bottom boundary of $\lambda$, then $\theta$ is a substrip of the outer
border strip $\theta_1$ of $\lambda/\mu$ going from $(a_j,b_j)$ to $(c_i,d_i)$,
\[
\#\{\text{ paths } \ga \mid  \ga
\subseteq \lambda, \ga:(a_j,b_j) \to (c_i,d_i)\,\} = \sED(\theta).
\]

Next, we claim that the substrip $\theta$ of $\theta_1$ described above is
precisely the substrip $\theta_i \# \theta_j$ of the cutting strip $\tau=\theta_1$ from
the Lascoux--Pragacz identity \eqref{eq:LPid}. This follows from
Lemma~\ref{lem:contentsLPvsK} since the starting point $(a_j,b_j)$ of
$\ga^*_j$ has the same content as the starting point of $\theta_j$ and
the end point $(c_i,d_i)$ of $\ga^*_i$ has the same content as the end
point of $\theta_i$, i.e. $p(\theta_j) = b_j-a_j$ and $q(\theta_i)=d_i-c_i$.
Thus $\theta = \phi[p(\theta_j), q(\theta_i)] = \theta_i \# \theta_j$
and so the previous equation becomes
\begin{equation} \label{eq:entries-LP-det}
\#\{\text{ paths } \ga \mid  \ga
\subseteq \lambda, \ga:(a_j,b_j) \to (c_i,d_i)\,\} = \sED(\theta_i \# \theta_j).
\end{equation}
Finally, the result follows by combining \eqref{eq:excited_2nddet} and \eqref{eq:entries-LP-det}.
\end{proof}

\begin{figure}
\includegraphics{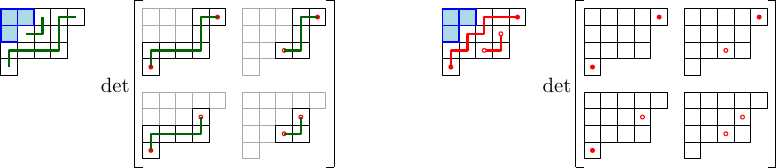}
\caption{For the shape $\lambda/\mu=(5441/21)$: on the left the
  Lascoux--Pragacz decomposition of $\lambda/\mu$  and the border strips of
  the determinantal identity for $s_{\lambda/\mu}$; on the right the
  Kreiman decomposition of $\lambda/\mu$ and the corresponding
  endpoints of the paths in the determinantal identity \eqref{eq:excited_2nddet} to calculate $\sED(\lambda/\mu)$.}
\label{fig:KvsLP}
\end{figure}
\bigskip

\section{Pleasant diagrams} \label{sec:pleasant}

\subsection{Definition and characterization}  \label{subsec:pleasant}

A {\em descending chain} of length $s$ of a set $S\subset [\lambda]$
is a sequence $\left((i_1,j_1),(i_2,j_2),\ldots,(i_s,j_s)\right)$ of
elements in $S$ satisfying  $1\leq i_1 < \cdots <
i_s \leq m$ and $1\leq j_1 < \cdots < j_s \leq n$.

\begin{definition}[Pleasant diagrams \cite{MPP1}] \label{def:agog}
 A diagram $S\subset [\lambda]$ is a {\em \pleasant~diagram} of $\lambda/\mu$ if for all
 integers $k$ with $1-\ell(\lambda)\leq k \leq \lambda_1-1$, the subarray
 $S_k := S\cap \square^{\lambda}_k$ has no descending chain bigger than
 the length of the $k$th diagonal of $\lambda/\mu$. We denote the set of
 \pleasant~diagrams of $\lambda/\mu$ by $\PD(\lambda/\mu)$ and its
 size by $\sPD(\lambda/\mu)$.
\end{definition}

\begin{example}
The pleasant diagrams of $\lambda/\mu = (22,1)$ are subsets of
$[\lambda]=\{(1,1),(1,2),(2,1),(2,2)\}$ with no descending chains in
$S_{-1},S_0,S_1$ of sizes $>1$. Thus, out of the $16$ subsets we discard the
four containing the descending chain $((1,1),(2,2))$, thus $\sPD(22/1)=12$.
\end{example}

Pleasant diagrams can be characterized in terms of complements of excited
diagrams.

\begin{theorem}[\cite{MPP1}] \label{thm:charpleasant}
A diagram $S\subset [\lambda]$ is a pleasant diagram in
$\PD(\lambda/\mu)$ if and only if $S \subseteq
[\lambda]\backslash D$ for some excited diagram $D\in \ED(\lambda/\mu)$.
\end{theorem}

Recall that by Proposition~\ref{prop:NIP2ED} for an excited diagram $D$, its complement corresponds to
a tuple of non-intersecting paths in $\NIP(\lambda/\mu)$ and that such
paths are characterized by their support of size $|\lambda/\mu|$.
Next, we give a formula for $\sPD(\lambda/\mu)$ from \cite{MPP1}. In
order to state it we need to define a peak statistic
 for the non-intersecting paths associated to the complement of an excited
diagram $D$.

To each tuple $\Gamma$ of non-intersecting paths we associate
recursively, via
ladder/excited moves, a subset of its support
called {\em excited peaks} and denoted by
$\Lambda(\Gamma)$. For $[\lambda/\mu] \in
\NIP(\lambda/\mu)$ the set of excited peaks is
$\Lambda([\lambda/\mu])=\varnothing$. If $\Gamma$ is a tuple in
$\NIP(\lambda/\mu)$ with an active cell $u=(i,j) \in
[\lambda]\setminus \Gamma$ then the excited peaks of
$\alpha_u(\Gamma)$ are
\[
\Lambda( \alpha_u(\Gamma)) := \Big(\Lambda(\Gamma)-
  \{(i,j+1),(i+1,j)\} \Big) \cup \{u\}.
\]
That is, the excited peaks of $\alpha_u(\Gamma)$ are obtained from those of
$\Gamma$ by adding the new peak $(i,j)$ and removing $(i,j+1)$ and $(i+1,j)$ if any of
the two are peaks in $\Lambda(\Gamma)$. Pictorially:
\begin{center}
\includegraphics{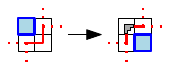},
\end{center}
where $\includegraphics{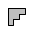}$ denotes a peak. Finally, let $\expeaks(\Gamma):=|\Lambda(\Gamma)|$ be
the number of excited peaks of $\Gamma$. Given
a set $\mathcal{S}$, let $2^{\mathcal{S}}$ denote the subsets of $\mathcal{S}$.

\begin{theorem}[\cite{MPP1}] \label{thm:num_pleasant}
For a skew shape $\lambda/\mu$ we have that $\PD(\lambda/\mu)$ is
given by the disjoint union
\[
\PD(\lambda/\mu) = \bigcup_{\Gamma \in \NIP(\lambda/\mu)} \Bigl(\Lambda(\Gamma) \times
2^{\Gamma \setminus \Lambda(\Gamma)}\Bigr).
\]
Thus
\[
\sPD(\lambda/\mu) = \sum_{\Gamma\in \NIP(\lambda/\mu)} 2^{|\lambda/\mu|-\expeaks(\Gamma)}.
\]
\end{theorem}

\begin{example} \label{ex:num_pleasant}
For the shape $\lambda/\mu = (5441/21)$, Figure~\ref{fig:excited} shows
the eight non-intersecting paths in $\NIP(\lambda/\mu)$, each with its
excited peaks marked by \includegraphics{peak}. Thus
\[
\sPD(5441/21) = 2^{11}+ 2\cdot 2^{10} + 2\cdot 2^9 + 2^8 + 2^{10}+ 2^9
= 6912.
\]
\end{example}

\subsection{Border strip decomposition formula for pleasant diagrams} \label{sec:LP4pleasant}

By Theorem~\ref{thm:num_pleasant}, the number of pleasant diagrams is
given by a weighted sum over non-intersecting paths in
$\NIP(\lambda/\mu)$. Since the number of such paths
$|\NIP(\lambda/\mu)| = \sED(\lambda/\mu)$ is given by a
Lascoux--Pragacz type determinant
(Theorem~\ref{thm:num_excited_HG}), one could ask if there also a similar determinantal
identity for $\sPD(\lambda/\mu)$. The following example shows that the
first guess of such a formula does not work. Later, we will see that Conjecture~\ref{conj:plesantdetSch}
suggests that in some cases there might be a determinantal formula for
$\sPD(\lambda/\mu)$.

\begin{example}
For $\lambda/\mu=(5441/21)$ we showed that $\sPD(5441/21)= 6912$,  but
$\sPD(5441/33)\sPD(22/1)- \sPD(322/11)\sPD(441/3) = 4352$ and the ratio of these
two numbers is $27/17$.
\end{example}

\begin{remark}
One difficulty in applying the Lindstr\"om--Gessel--Viennot Lemma
(Lemma~\ref{lem:LGVlambda}) in order to
write $\PD(\lambda/\mu)$ as a determinant of $\PD(\theta_i\#\theta_j)$
is that the non-intersecting paths corresponding to a pleasant diagram have
excited-peaks that depend on the structure of the path and not just
on the coordinates of its support. In the proof of the Lemma, the
sign-reversing involution of switching the
paths that intersect will not respect these
{\em local} excited peaks.
\end{remark}






\bigskip

\section{Combinatorial proofs of the NHLF and first $q$-NHLF} \label{sec:all_shapes}

The goal of this section is to give a combinatorial proof of the NHLF.
The proof is split into two parts: first, we reduce the claim from
all skew shapes to the border strips.  We then give two elementary
proofs of the NHLF in the border strip case, in the two sections that
follow.


\subsection{NHLF for border strips}  In this case the NHLF
is more elegant. We give two proofs of this case in
Sections~\ref{sec:1stproof} and \ref{sec:2ndproof}.

\begin{lemma}[NHLF for border strips] \label{thm:NHLFborderstrips}
For a border strip $\theta=\lambda/\mu$ with endpoints $(a,b)$ and
$(c,d)$ we have
\begin{equation} \label{eq:NHLFborderstrips}
\frac{f^{\theta}}{|\theta|!} = \sum_{\substack{\ga:  (a,b) \to (c,d),\\\ga \subseteq \lambda} }
\prod_{(i,j) \in \ga} \frac{1}{h(i,j)},
\end{equation}
where $h(i,j) = \lambda_i-i + \lambda'_j-j+1$.
\end{lemma}


Since the endpoints $(a,b)$ and $(c,d)$ are on the boundary of
$\lambda$ without loss of generality we assume that
$(a,b)=(\lambda'_1,1)$ and $(c,d)=(1,\lambda_1)$. The proof is
based on an identity of the following multivariate function. For a border strip $\lambda/\mu$ let
\[
F_{\lambda/\mu}({\bf x}\mid{\bf y}) =
F_{\lambda/\mu}(x_1,x_2,\ldots,x_d \mid
y_1,y_2\ldots,y_{n-d}):=\sum_{\substack{\ga:(\lambda'_1,1) \to
    (1,\lambda_1),\\\ga\subseteq \lambda}}
\prod_{(i,j)\in \ga} \frac{1}{x_i-y_j}.
\]
Note if we evaluate $F_{\lambda/\mu}({\bf x}\mid {\bf y})$ at
$x_i =\lambda_i+d-i+1$ and $y_j = d+j-\lambda'_j$ we obtain the RHS of \eqref{eq:NHLFborderstrips},
\begin{equation} \label{eq:hooksEDborder-strip}
\left.F_{\lambda/\mu}({\bf x} \mid {\bf y})
\right|_{\substack{x_i=\lambda_i+d-i+1,\\y_j=d+j-\lambda'_j}} =
\sum_{\substack{\ga:(\lambda'_1,1) \to
    (1,\lambda_1),\\\ga\subseteq \lambda}} \prod_{(i,j)\in \ga} \frac{1}{h(i,j)}.
\end{equation}

\subsection{From border trips to all skew shapes}
We need the analogue of  Theorem~\ref{thm:LascouxPragacz} for $f^{\lambda/\mu}$.

\begin{lemma}[Lascoux--Pragacz]
\label{lem:LascouxPragacz-SYT}
If $(\theta_1,\ldots,\theta_k)$ is a Lascoux--Pragacz decomposition of
$\lambda/\mu$, then
\begin{equation} \label{eq:LPid4SYT}
\frac{f^{\lambda/\mu}}{|\lambda/\mu|!} \,=\,  \det\left[\,  \frac{f^{\theta_i
      \# \theta_j}}{|\theta_i \# \theta_j|!} \,\right]_{i,j=1}^k.
\end{equation}
where $f^{\varnothing}=1$ and $f^{\phi[p,q]} =0$ if $\phi[p,q]$ is undefined.
\end{lemma}

\begin{proof}
The result follows by doing the stable principal specialization in
\eqref{eq:LPid}, using the theory of $P$-partitions \cite[Thm. 3.15.7]{EC2} and
letting $q\to 1$.
\end{proof}

\begin{proof}[Proof of Theorem~\ref{thm:IN}]
Combining Lemma~\ref{thm:NHLFborderstrips} and
Lemma~\ref{lem:LascouxPragacz-SYT} we have
\begin{equation} \label{eq1:pfNHLF}
f^{\lambda/\mu} = |\lambda/\mu|! \cdot \det\left[
 \sum_{\substack{\ga:  (a_j,b_j) \to (c_i,d_i),\\ \ga \subseteq \lambda}}
\prod_{(r,s) \in \ga} \frac{1}{h(r,s)}
\right]_{i,j=1}^k.
\end{equation}
Note that the weight $1/h(r,s)$ of each step in the path only depends on
the coordinate $(r,s)$ and the fixed partition $\lambda$. By the
weighted Lindstr\"om--Gessel--Viennot lemma (Lemma~\ref{lem:LGVlambda}),
with $y_{r,s}=1/h(r,s)$, we rewrite the RHS of
\eqref{eq1:pfNHLF} as a weighted sum over  $k$-tuples of non-intersecting
paths $\Gamma$ in $\NIP(\lambda/\mu)$. That is,
\begin{equation}
f^{\lambda/\mu} = |\lambda/\mu|! \cdot \sum_{(\ga_1,\ldots,\ga_k) \in \NIP(\lambda/\mu)}
\prod_{(r,s)\in (\ga_1,\ldots,\ga_k)} \frac{1}{h(r,s)},
\end{equation}
Finally, by Proposition~\ref{prop:NIP2ED} the supports of these
non-intersecting paths are precisely the complements of excited
diagrams of $\lambda/\mu$. This finishes the proof of NHLF.
\end{proof}

\subsection{Proof of the first $q$-NHLF}
In this case too the SSYT $q$-analogue of NHLF is elegant and can be
stated as follows.

\begin{lemma} \label{lem:qNHLFborderstrips}
For a border strip $\theta=\lambda/\mu$ with end points $(a,b)$ and
$(c,d)$ we have
\begin{equation} \label{eq:qschurborderstrip}
s_{\theta}(1,q,q^2,\ldots,) = \sum_{\substack{\ga:  (a_j,b_j) \to (c_i,d_i),\\ \ga \subseteq \lambda}}
\prod_{(i,j) \in \ga} \frac{q^{\lambda'_j-i}}{1-q^{h(i,j)}}.
\end{equation}
\end{lemma}

The proof is postponed to Section~\ref{ss:qborderstrip}.


\begin{lemma}[Lascoux--Pragacz] \label{lem:LascouxPragacz-SSYT}
\begin{equation} \label{eq:LascouxPragacz-SSYT}
s_{\lambda/\mu}(1,q,q^2,\ldots)
= \det \begin{bmatrix}  s_{\theta_i \#\theta_j}(1,q,q^2,\ldots)  \end{bmatrix}_{i,j=1}^k,
\end{equation}
where $s_{\varnothing}=1$ and $s_{\phi[p,q]}=0$ if $\theta[p,q]$ is undefined.
\end{lemma}

\begin{proof}
The result follows by doing a stable principal specialization in \eqref{eq:LPid}.
\end{proof}

\begin{proof}[Proof of Theorem~\ref{thm:skewSSYT}]
Combining Lemma \ref{lem:LascouxPragacz-SSYT} and Lemma~\ref{lem:qNHLFborderstrips} we have
\begin{equation} \label{eq1:pf-qNHLF}
s_{\lambda/\mu}(1,q,q^2,\ldots) = \det\left[
 \sum_{\substack{\ga:  (a_j,b_j) \to (c_i,d_i),\\ \ga \subseteq \lambda}}
\prod_{(r,s) \in \ga} \frac{q^{\lambda'_s-r}}{1-q^{h(r,s)}}
\right]_{i,j=1}^k.
\end{equation}
Note that the weight of each step $(r,s)$ in the path is
$q^{\lambda'_s-r}/(1-q^{h(r,s)})$ which only depends on
the coordinate $(r,s)$ and the fixed partition $\lambda$.
By the weighted Lindstr\"om--Gessel--Viennot lemma
(Lemma~\ref{lem:LGVlambda}), with $y_{r,s} = q^{\lambda'_s-r}/(1-q^{h(r,s)})$, we rewrite the RHS of
\eqref{eq1:pf-qNHLF} as a weighted sum of $k$-tuples of non-intersecting
paths in $[\lambda]$. That is,
\begin{equation}
s_{\lambda/\mu}(1,q,q^2,\ldots)  = \sum_{(\ga_1,\ldots,\ga_k) \in \NIP(\lambda/\mu)}
\prod_{(r,s)\in (\ga_1,\ldots,\ga_k)} \frac{q^{\lambda'_s-r}}{1-q^{h(r,s)}},
\end{equation}
Finally, by Proposition~\ref{prop:NIP2ED} the supports of these
non-intersecting paths are precisely the complements of excited
diagrams of $\lambda/\mu$. Thus we obtain Equation \eqref{eq:skewschur}.
\end{proof}

\bigskip

\section{First proof of NHLF for border strips} \label{sec:1stproof}

In this section we give a proof of the NHLF for border strips
based on a multivariate identity of the weighted
sum of paths $F_{\theta}({\bf x}\mid {\bf y})$. We show that this
weighted sum satisfies a recurrence for SYT.

\subsection{Multivariate lemma}

For any connected skew shape $\lambda/\mu$, the entry $1$ in a standard Young tableau
$T$ of shape $\lambda/\mu$ will be in an inner corner of $\lambda/\mu$. The remaining entries
$2,3,\ldots,n$ form a standard Young tableau $T'$ of shape
$\lambda/\nu$ where $\mu \to \nu$. Conversely, given a standard
Young tableau $T'$ of shape $\lambda/\nu$ where $\mu \to \nu$, by
filling the new cell with $0$ we
obtain a standard Young tableau of shape $\lambda/\mu$. Thus
\begin{equation} \label{eq:ChevalleySYT}
f^{\lambda/\mu} = \sum_{\mu \to \nu} f^{\lambda/\nu}.
\end{equation}

We show combinatorially that for border strips $\lambda/\mu$
the multivariate rational function $F_{\lambda/\mu}({\bf x}\mid {\bf
  y})$ satisfies
this type of relation.

\begin{lemma}[Pieri--Chevalley formula for border strips] \label{lem:ChevalleyStrips}
\begin{equation} \label{eq:ChevalleyStrips}
F_{\lambda/\mu}({\bf x}\mid {\bf y}) = \frac{1}{x_1-y_1}\sum_{\mu \to \nu} F_{\lambda/\nu}({\bf x} \mid {\bf y}).
\end{equation}
\end{lemma}


\begin{remark}
A very similar multivariate relation holds for general skew shapes
(the only difference is a different linear factor on the RHS of \eqref{eq:ChevalleyStrips}), a fact proved by
Ikeda and Naruse \cite{IkNa09} algebraically and combinatorially by
Konvalinka~\cite{Ko}. Our proof for border strips is different
than these two proofs. See Section~\ref{ss:comparisonKonvalinka} for more details.
\end{remark}

\subsection{Proof of multivariate lemma}
The rest of the section is devoted to the proof of Lemma~\ref{lem:ChevalleyStrips}.
We start with some notation that will help us in the proof.

For cells $A, B \in [\lambda]$ such that $B$ is northeast of~$A$, let
\[
F(A\to B) := \sum_{\ga:A\to B, \ga \subseteq [\lambda]} \prod_{(i,j)
  \in \ga} \frac{1}{x_i-y_j},
\]
so that $F_{\lambda/\mu}({\bf x} \mid {\bf y}) = F((\lambda'_1,1)\to
(1,\lambda_1))$. For a given path $\gamma$ let
$$H(\gamma) := \prod_{(i,j) \in \ga} \frac{1}{x_i-y_j}$$ be its multivariate weight.  Let $F(A^*,B)$ and $F(A,B^*)$ denote similar rational
functions where we omit the term $x_i-y_j$ corresponding to $A$ and
$B$ respectively. In particular, $F(A^*,A)=F(A,A^*)=1$.  By abuse of notation $F(A\to C^*\to B)$ denotes
the product $F(A\to C^*)F(C^*\to B)$. Let $\overline{C}$ and
$\underline{C}$ denote the boxes in the Young diagram $[\lambda]$ that
are immediately above and below $C$, respectively. Let $R_k(\lambda)$ denote the $k$th row of the Young diagram of $\lambda$.

We will show that
\begin{equation} \label{eq:pathstarget}
F(A\to B) = \frac{1}{x_1-y_1}\sum_{C} F(A\to C^*\to B),
\end{equation}
where the sum is over inner corners $C$ of $\lambda/\mu$. This relation
 implies the desired result.

\begin{example}
Consider $\lambda=(2,2)$ and $\mu=(1)$, the shape $(2,2)/(1)$
has inner corners $(1,2)$ and $(2,1)$. We have
\begin{align*}
F((2,1)\to (1,2)) &= \frac{1}{(x_2-y_1)(x_2-y_2)(x_1-y_2)} + \frac{1}{(x_2-y_1)(x_1-y_1)(x_1-y_2)}\\
&= \frac{x_1-y_1 + x_2-y_2}{(x_1-y_1)(x_2-y_1)(x_1-y_2)(x_2-y_2)}\\
&= \frac{1}{x_1-y_1} \left( \frac{1}{(x_2-y_1)(x_2-y_2)} +
  \frac{1}{(x_2-y_2)(x_1-y_2)}\right),
\end{align*}
which equals $\left[F((1,2)\to (2,1)^*) + F((1,2)^*\to
  (2,1))\right]/(x_1-y_1)$, thus proving the relation.
\end{example}

We prove \eqref{eq:pathstarget} by induction on the total length of the
path between $A$ and $B$. The base case  $\lambda=(1)$ and $\mu=\varnothing$ since
$F(A\to B)=1/(x_1-y_1)$ and $(1,1)$ is the only inner corner so
$F(A\to C^*\to B)=1$. The next sublemma will be useful in the inductive step later.

%

\begin{lemma} \label{proof_Nborders_lem2}
For cells $A=(d,r)$ and $B=(1,s)$ in $[\lambda]$ with $r\leq s$, we have
\[
(x_1-x_d) F(A\to B) = \sum_C F(A\to C^*\to B),
\]
where the sum is over inner corners $C$ of $\lambda/\mu$.
\end{lemma}

\begin{proof}

We can write $x_k - x_{k-1} = (x_k - y_j) - (x_{k-1} - y_j)$ for any $j$. Let $\gamma$ be a path from $A$ to $B$, and suppose that it crosses from row $k$ to row $k-1$ in column $j$ for some $j$. Then both points $(k,j) \in \gamma$ and $(k-1,j) \in \gamma$
\begin{equation} \label{eq:pfspaths}
(x_k - x_{k-1}) H(\gamma) = (x_k - y_j)H(\gamma) - (x_{k-1} -
y_j)H(\gamma) = H(\gamma \setminus (k,j) ) - H(\gamma \setminus
(k-1,j) ).
\end{equation}
Since every path from $A$ to $B$ crosses from row $k$ to row $k-1$ at
some cell, denoted by $C = (k,j)$, by \eqref{eq:pfspaths} then we have the following:
\begin{align}
(x_k - x_{k-1})F(A \to B) \, = \, \sum_{C \in R_k(\lambda)} \Big(F(A \to C^*) F(\overline{C} \to B) - F(A \to C)F(\overline{C}^* \to B) \Big) \notag
\\
=\, \sum_{C \in R_k(\lambda)} F(A \to C^* \to \overline{C} \to B) - \sum_{C_1 \in R_{k-1}(\lambda)} F(A \to \underline{C_1} \to C_1^* \to B) \, =: \ \textcircled{$\ast$} \notag
\end{align}
where in the last line we denote $C_1 = \overline{C}$ -- a box in row
$k-1$. Note that the existence of the boxes below and above is implicit in the specified path functions $F$.

Let us now rewrite the RHS. in the last equation in a different
way. Note that the paths $A \to C^* \to \overline{C} \to B$ can be
thought of as paths from $A$ to $B$ without their outer corner on row
$k$, and, likewise, the paths  $A \to \underline{C_1} \to C_1^* \to B$
are paths $A \to B$ without the inner corner on row $k-1$. However,
they can both be thought of as composed of two paths, $A \to A_1$ and
$B_1 \to B$, where $A_1$ is the last box on row $k$ (or row $k+1$ if
$C$ was the only cell on row $k$), $B_1$ is the first box on row $k-1$ (or the box above $\overline{C}$, in row $k-2$) and $A_1$'s top right vertex is the same as $B_1$'s bottom left (i.e. the boxes have that common vertex), or as in the second case $B_1$ is one box above $A_1$. In the case of $A \to C^* \to \overline{C} \to B = A \to A_1, B_1 \to B$, we must have that $A_1$ is not the last box in the row (for $C$ to exist), and for $A \to \underline{C_1} \to C_1^* \to B = A\to A_1, B_1\to B$ there are no restrictions. Thus
\begin{align*}
\lefteqn{(x_k - x_{k-1})F(A \to B) \, = \ \textcircled{$\ast$} }\\
&= \sum_{A_1 \neq (k,\lambda_k) , B_1} F(A \to A_1)F(B_1\to B) \,-\, \sum_{A_1 , B_1} F(A \to A_1)F(B_1\to B) \\
&= \sum_{j: \substack{ A_1= (k+1,j),\\ B_1=(k-1,j)} } F(A \to
  A_1)F(B_1\to B)\, - , \sum_{j: \substack{ A_1 = (k,j),\\ B_1=(k-2,j)}}
  \Bigl( F(A \to A_1)F(B_1\to B) - F(A \to D_k^* \to B) \Bigr),
\end{align*}
where all terms cancel except for the cases where $A_1,C,B_1$ are in the same column, and when $C$ is an outer corner of $\lambda$ on row $k$, denoted by $D_k$ (if such corner exists).

Finally, since $x_d - x_1 = \sum_k (x_k-x_{k-1})$, we have
\begin{align*}
\lefteqn{ (x_d-x_1)F(A \to B) \, = \, \sum_k (x_k-x_{k-1})F(A \to B) }\\
&= \sum_{k, j: \substack{ A_1= (k+1,j),\\ B_1=(k-1,j)} } F(A \to
  A_1)F(B_1\to B) \ - \, \sum_{k, j: \substack{ A_1 = (k,j),\\ B_1=(k-2,j)}}
  \Bigl( F(A \to A_1)F(B_1\to B) - F(A \to D_k^* \to B) \Bigr)\\
&= \, -\.\sum_k F(A \to D_k^* \to B),
\end{align*}
since all other terms cancel across the various values for $k$, and we obtain the desired identity. \end{proof}

%

We continue with the proof of Lemma~\ref{lem:ChevalleyStrips}. In a path
$\ga:A\to B$ the first step from $A$ is either
right to cell $A_r$ or up to cell $A_u$. Note that in the first case
$A$ is then an inner corner of $\lambda/\mu$. Thus
\[
F(A\to B) \, =\, \frac{1}{x_d-y_1} \, \left( F(A_r\to B) + F(A_u\to B) \right).
\]
By induction the term $F(A_u\to B)$ becomes
\begin{equation} \label{eq:pathsA}
F(A\to B)\, =\, \frac{1}{x_d-y_1} \left( F(A_r \to B)\. +\. \frac{1}{x_1-y_1} \.
  \sum_C F(A_u\to C^*\to B)\right).
\end{equation}

On the other hand, since a step to $A_r$ indicates that $A$ is an
inner corner then the RHS of
\eqref{eq:pathstarget} equals
\[
\frac{1}{x_1-y_1}\sum_{C} F(A\to C^*\to B) \, = \, \frac{1}{x_1-y_1}\left[
  F(A_r \to B) \.+ \. \frac{1}{x_d-y_1}\sum_{C} F(A^* \to C^* \to B) \right].
\]
Again, depending on the first step of the paths we split $F(A^* \to C^* \to B)$
into $F(A_r \to C^* \to B)$ and $F(A_u\to C^* \to B)$ so the above
equation becomes
\begin{multline} \label{eq:pathsB}
\frac{1}{x_1-y_1}\sum_{C} F(A\to C^*\to B) \\ = \,
\frac{1}{x_1-y_1} \left[F(A_r \to B) \.+ \. \frac{1}{x_d-y_1}\sum_C\Bigl(
    F(A_r\to C^*\to B) \. + \. F(A_u \to C^*\to B)\Bigr)\right].
\end{multline}

Finally, by \eqref{eq:pathsA} and \eqref{eq:pathsB} if we subtract the
LHS and RHS of \eqref{eq:pathstarget} the terms with $A_u \to
C^* \to B$ cancel. Collecting the terms with $A_r \to B$ we obtain
\begin{multline}
F(A\to B) \. - \, \frac{1}{x_1-y_1}\.\sum_{C} F(A\to C^*\to B) \\ =\,
\frac{x_1-x_d}{(x_1-y_1)(x_d-y_1)}\. F(A_r \to B)\.
  - \, \frac{1}{(x_1-y_1)(x_d-y_1)} \.\sum_C F(A_r\to C^*\to B_1)\ts.
\end{multline}
Lastly, the RHS above is zero since by Lemma~\ref{proof_Nborders_lem2}
we have
\[
(x_1-x_d)F(A_r\to B)\, = \, \sum_C F(A_r \to C^* \to B).
\]
Thus the desired relation \eqref{eq:pathstarget} follows.

\subsection{Proof of NHLF for border strips}

In this section we use Lemma~\ref{lem:ChevalleyStrips} to prove
Theorem~\ref{thm:NHLFborderstrips}.

Let $H_{\lambda/\mu}$ denote the RHS of
\eqref{eq:hooksEDborder-strip}. We prove by induction on
$n=|\lambda/\mu|$ that $f^{\lambda/\mu} = n! \cdot H_{\lambda/\mu}$.

We start with \eqref{eq:ChevalleyStrips} and evaluate $x_i =
\lambda_i+d-i+1$ and $y_j=d+j-\lambda'_j$, by
\eqref{eq:hooksEDborder-strip} we obtain
\[
n\cdot H_{\lambda/\mu} = \sum_{\mu \to \nu} H_{\lambda/\nu}.
\]
Multiplying both sides by $(n-1)!$ and using induction we obtain
\[
n! \cdot H_{\lambda/\mu} = \sum_{\mu \to \nu} f^{\lambda/\nu}.
\]
By \eqref{eq:ChevalleySYT} the result follows.

\bigskip

\section{Second proof of NHLF for border strips} \label{sec:2ndproof}
In this section we give another proof of the NHLF for border strips
based on another multivariate identity involving {\em factorial Schur
  functions}. The proof consists of two steps. First we show that a
ratio of an evaluation of factorial Schur functions equals the weighted
sum of paths $F_{\theta}({\bf x}\mid {\bf y})$. Second we show
how the ratio of factorial Schur functions properly specialized equals
$f^{\lambda/\mu}$ and $s_{\lambda/\mu}(1,q,q^2,\ldots)$.

\subsection{Multivariate lemma}
We show combinatorially that the function $F_{\lambda/\mu}({\bf x} \mid {\bf
  y})$ is an evaluation of a factorial Schur function. Let ${\bf
  z}^{\lambda}$ be the word of length $n$ of $x$'s and $y$'s obtained by
reading the horizontal and vertical steps of $\lambda$ from $(d,1)$ to
$(1,n-d)$:  i.e. $z^{\lambda}_{\lambda_i+d-i+1} = x_i$ and $z^{\lambda}_{\lambda'_j+n-d-j+1}=y_j$:
\begin{center}
\includegraphics{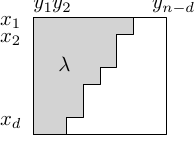}
\end{center}

\begin{example}
For $d=4, n=9$ and $\lambda = (5533)$ we have that
$z^{\lambda} = y_1y_2y_3x_4x_3y_4y_5x_2x_1$. See the figure in the
next example.
\end{example}

\begin{lemma}[\cite{IkNa09}] \label{lem:key_border_strips}
For a border strip $\lambda/\mu \subseteq d\times (n-d)$ we have
\begin{equation} \label{eq:key_border_strips}
\frac{s_{\mu}^{(d)}({\bf x} \mid {\bf
    z}^{\lambda})}{s_{\lambda}^{(d)}({\bf x} \mid {\bf
    z}^{\lambda})} =
F_{\lambda/\mu}({\bf x} \mid {\bf y}).
\end{equation}
\end{lemma}

Before we begin the proof we make a few definitions to simplify notation and a few observations to be used throughout.
For any partition $\nu \subseteq d\times (n-d)$ and a set of variables ${\bf x}$ and ${\bf z}$ define
$$D(\nu):= \det[ (x_i-z_1)\ldots(x_i-z_{\nu_j+d-j}) ]_{i,j=1}^d,$$
so that
\begin{equation} \label{eq:GratioD}
\frac{s_{\mu}^{(d)}({\bf x} \mid {\bf
      z}^{\lambda})}{s_{\lambda}^{(d)}({\bf x} \mid {\bf z}^{\lambda})} = \frac{D(\mu)}{D(\lambda)}.
\end{equation}
Notice also that $z_{\lambda_j+1+d-j}=x_j$ and so $ (x_i-z_1)\ldots(x_i-z_{\lambda_j+d-j})  =0$ if
$j<i$. So the matrix in $D(\lambda)$ is upper-triangular and
\begin{equation}\label{eq:D_upper}
D(\lambda) =
\prod_{i=1}^d (x_i-z_1)\cdots(x_i -z_{\lambda_i+d-i}).
\end{equation}

\subsection{Proof of multivariate lemma}

To prove Lemma~\ref{lem:key_border_strips} we verify that both sides of \eqref{eq:key_border_strips} satisfy
the following trivial path identity. The first step of a  path $\ga:(\lambda'_1,1) \to
  (1,\lambda_1)$ is either $(0,1)$ (up) or $(1,0)$ (right) provided $\lambda_d>1$. So
\begin{equation} \label{eq:FpathRel}
(x_d-y_1)F_{\lambda/\mu}({\bf x}\mid  {\bf y}) =
F_{\lambda-\lambda_d/\mu-\mu_{d-1}}(x_1,\ldots,x_{d-1}\mid {\bf y}) + F_{\lambda - {\bf 1}/\mu-{\bf
    1}}({\bf x}\mid y_2,\ldots,y_{n-d}),
\end{equation}
where the second term on the RHS vanishes if $\lambda_d=1$.

\begin{example}
For the border strip $\lambda/\mu = (5533/422)$, we have
\begin{multline*}
(x_4-y_1) F_{(5533/422)}(x_1,\ldots,x_4 \mid y_1,\ldots,y_5) =\\
F_{(553/42)}(x_1,x_2,x_3\mid y_1,\ldots,y_5) + F_{(4422/311)}(x_1,\ldots,x_4\mid y_2,\ldots,y_5),
\end{multline*}
\begin{center}
\includegraphics{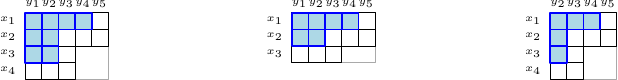}
\end{center}

\end{example}

Next we show that the following ratio of factorial Schur functions, satisfies the
same relation:
\[
G_{\lambda/\mu}({\bf x} \mid {\bf y}) := \frac{s_{\mu}^{(d)}({\bf x} \mid {\bf
      z}^{\lambda})}{s_{\lambda}^{(d)}({\bf x} \mid {\bf z}^{\lambda})}.
\]

\begin{lemma} \label{lem:GpathRel}
We have: 
\begin{equation} \label{eq:lemGpathRel}
(x_d-y_1)G_{\lambda/\mu}({\bf x}\mid  {\bf y}) \, = \,
G_{\lambda-\lambda_d/\mu-\mu_{d-1}}(x_1,\ldots,x_{d-1}\mid {\bf
  y}) \. +\. G_{\lambda - {\bf 1}/\mu-{\bf
    1}}({\bf x}\mid y_2,\ldots,y_{n-d})\ts,
\end{equation}
where the second term on the RHS vanishes if $\lambda_d=1$.
\end{lemma}

\begin{proof}[Proof of Lemma~\ref{lem:key_border_strips}]
We proceed by induction. For the base case $\lambda=(1)$
and $\mu=\varnothing$, we directly check that
\[
F_{(1)/\varnothing}({\bf x}\mid {\bf y}) \. = \. G_{(1)/\varnothing}({\bf x}
\mid {\bf y}) = \frac{1}{x_d-y_1}\ts.
\]
Then by \eqref{eq:FpathRel} and Lemma~\ref{lem:GpathRel} we have
$F_{\lambda/\mu}({\bf x} \mid {\bf y})$ and $G_{\lambda/\mu}({\bf x}\mid {\bf
  y})$ satisfy the same recurrence. Therefore, we have
$F_{\lambda/\mu}({\bf x} \mid {\bf y}) = G_{\lambda/\mu}({\bf x}\mid
{\bf y})$ as desired.
\end{proof}

In the rest of the section we prove Lemma~\ref{lem:GpathRel}.

\begin{proof}[Proof of Lemma~\ref{lem:GpathRel}]
We denote the shape $(\lambda-\lambda_d) / (\mu-\mu_{d-1})$ by $\overline{\lambda}/\overline{\mu}$.
Removing the first column of $\lambda$ yields
${\bf z}^{\lambda-{\bf 1}} =y_2,\ldots=z_2,\ldots$ and 
removing the last row of
$\lambda$ yields ${\bf z}^{\bar{\lambda}} = y_1,y_2,\ldots,\widehat{x_d},\ldots =
z_1,\ldots,z_{\lambda_d},z_{\lambda_d+2},\ldots$, i.e. ${\bf z^{\lambda}}$ with the
entry $x_d$ omitted.

Assume $\lambda_d\neq 0$ and $\mu_d=0$, the other case is trivially
reduced.  If $\lambda/\mu$ is a border strip $\mu_j =\lambda_{j+1} -1$
for $j=1,\ldots,d-1$. Hence in the ratio of determinants in \eqref{eq:GratioD} we have
that the first $d-1$ columns of the determinant from
$s^{(d)}_{\mu}(\cdot \mid \cdot )$ are the last $d-1$ columns from the
determinant for $s^{(d)}_{\lambda}(\cdot \mid \cdot)$, and the $d$th
column from $s^{(d)}_{\mu}(\cdot \mid \cdot)$ is all ones, since
$\mu_d +d - d =0$. Thus in \eqref{eq:GratioD}, upon shifting the $d$th column to the first
column of the determinant $D(\mu)$ in the numerator, we obtain
 \begin{equation} \label{eq:Forig}
 G_{\lambda/\mu}({\bf x} \mid {\bf y})  = \frac{ (-1)^{d-1}}{D(\lambda)}\det\left
   [ \begin{cases} 1, & j=1\\  (x_i-z_1)\ldots(x_i-z_{\lambda_j+d-j}),
     &  j>1 \end{cases} \right]_{i,j=1}^d.
 \end{equation}
Next we have two cases depending on whether $\lambda_d=1$ or
$\lambda_d>1$.

\medskip

\noindent {\em Case $\lambda_d=1$:} For $\lambda$ we have
$z^{\lambda}_1=y_1$ and $z^{\lambda}_2=x_d$. The $(d,d)$ entry of the upper triangular
matrix of the determinant in $D(\lambda)$ is $x_d-y_1$, so by doing a
cofactor expansion on this row we get
\[
D(\lambda) \, =\, (x_d-y_1) \ts \det
\bigl[(x_i-x_d)(x_i-y_1)(x_i-z_3)\cdots (x_i-z_{\lambda_j+d-j})\bigr]_{i,j=1}^{d-1}\ts.
\]
By factoring $x_i-z_2=x_i-x_d$ from each row above we get
\[
D(\lambda) \, =  \, (x_d-y_1) \ts
             \det\left[(x_i-y_1)(x_i-z_3)\cdots
             (x_i-z_{\lambda_j+d-j})\right]_{i,j=1}^{d-1} \prod_{i=1}^{d-1}(x_i-x_d)\ts.
\]
Since $z^{\overline{\lambda}}_1=y_1$ and
$z^{\overline{\lambda}}_j=z^{\lambda}_{j+1}$ for $j=2,\ldots,d-1$,
then by relabeling we get
\begin{equation} \label{eq:case1dlam}
D(\lambda) \. = \. (x_d-y_1)\ts D(\overline{\lambda}) \. \prod_{i=1}^{d-1}(x_i-x_d)\ts.
\end{equation}

For $\mu$ we have $\mu_d=\mu_{d-1}=0$ so the matrix in $D(\mu)$
has a $d$th column of ones
\[
D(\mu) = \det \begin{bmatrix} \ldots & (x_1-z_1)\cdots(x_1 -
  z_{\mu_j+d-j}) & \cdots & (x_1-y_1) & 1 \\ \ldots & (x_2-z_1)\cdots(x_2 -
  z_{\mu_j+d-j}) &\cdots  & (x_2-y_1) & 1 \\ \vdots & &  & & \vdots \\
  0 & \cdots & 0 & (x_d-y_1) & 1 \end{bmatrix}
\]
Then, by adding $(x_d-y_1)$ to each entry in the $(d-1)$th column of
the matrix above, the determinant
remains unchanged but the last row becomes $0\ldots 01$. This gives
\[
D(\mu) = \det \begin{bmatrix} \begin{cases} 1, &  j=d\\ x_i-x_d, &
    j=d-1\\
 (x_i-z_1)\ldots(x_i-z_{\mu_j+d-j}),
     &  j<d-1 \end{cases}\end{bmatrix}_{i,j=1}^d
\]
Next, we do a cofactor expansion on the last row of this matrix and then we
factor $x_i-z_2=x_i-x_d$ from each row,
\begin{align*}
D(\mu) &= \det \begin{bmatrix} \begin{cases} x_i-x_d, &
    j=d-1\\
 (x_i-z_1)(x_i-z_2)\ldots(x_i-z_{\mu_j+d-j}),
     &  j<d-1 \end{cases} \end{bmatrix}_{i,j=1}^{d-1}\\
&= \det \begin{bmatrix} \begin{cases} 1, &
    j=d-1\\
 (x_i-z_1)\widehat{(x_i-x_d)}\ldots(x_i-z_{\mu_j+d-j}),
     &  j<d-1 \end{cases} \end{bmatrix}_{i,j=1}^{d-1}
       \prod_{i=1}^{d-1} (x_i-x_d).
\end{align*}
Again, since $z^{\bar{\lambda}}_1=y_1$ and
$z^{\bar{\lambda}}_j=z^{\lambda}_{j+1}$ for $j=2,\ldots,d-1$, we have
by relabeling that
\begin{equation} \label{eq:case1dmu}
D(\mu) = D(\bar{\mu}) \prod_{i=1}^{d-1} (x_i-x_d).
\end{equation}

We now combine \eqref{eq:case1dlam} and \eqref{eq:case1dmu} in
$(x_d-y_1)G_{\lambda/\mu}(\cdot \mid \cdot)$,
\[
(x_d-y_1)
  G_{\lambda/\mu}({\bf x}\mid {\bf y})  = (x_d-y_1) \frac{
    D(\bar{\mu})\prod_{i=1}^{d-1} (x_i-x_d) }{ (x_d-y_1)
     D(\bar{\lambda})\prod_{i=1}^{d-1} (x_i - x_d) } =
  G_{\bar{\lambda}/\bar{\mu}}(x_1,\ldots,x_{d-1}\mid {\bf y} ),
\]
confirming the desired identity \eqref{eq:lemGpathRel} in this case since the term
$G_{\lambda-{\bf 1}/\mu-{\bf 1}}({\bf x} \mid {\bf y})$ is vacuously
  zero when $\lambda_d=1$.

\medskip

\noindent {\em Case $\lambda_d>1$:} Using $z_1=y_1$ we have
\begin{align}
 G_{\lambda-{\bf 1}/\mu-{\bf 1}}({\bf x}\mid y_2,\ldots,y_{n-d}) &=
 \frac{(-1)^{d-1}}{D(\lambda - {\bf 1})} \det\left [ \begin{cases} 1, &  j=1\\
       (x_i-z_2)\ldots(x_i-z_{\lambda_j+d-j}), &  j>1 \end{cases}
                                                \right]_{i,j=1}^d
                                                 \notag \\
&=  \frac{(-1)^{d-1}}{D(\lambda)}\det\left [ \begin{cases} 1, &  j=1\\
    (x_i-z_2)\ldots(x_i-z_{\lambda_j+d-j}), &  j>1 \end{cases}
                                             \right]_{i,j=1}^d \cdot
                                              \prod_{i=1}^d (x_i-y_1)
                                              \notag \\
& = \frac{(-1)^{d-1}}{D(\lambda)} \det\left [ \begin{cases} (x_i-y_1),
    & j=1\\ (x_i-z_1) \ldots(x_i-z_{\lambda_j+d-j}),&  j>1 \end{cases}
                                                      \right]_{i,j=1}^d \label{eq:Flam_minus1}
 \end{align}
Similarly, we have:
\begin{multline} \label{eq:F_bar}
 G_{\overline{\lambda}/\overline{\mu} }(x_1,\ldots,x_{d-1} \mid {\bf y}) \\
 =\, \frac{ (-1)^{d-2}}{D(\lambda)} \. \det\left [ \begin{cases}  (x_i -x_d),
     &j=1 \\ (x_i-z_1) \ldots(x_i-z_{\lambda_{j-1}+d-{j-1}}), &2 \leq
     j \leq d-1 \end{cases}\right ]_{i,j=1}^{d-1}  \cdot \prod_{j=1}^{\lambda_d}(x_d-y_j)
%
 \end{multline}

 Next, we evaluate the difference of $(x_d-y_1)G_{\lambda/\mu}(\cdot \mid
 \cdot )$ and  $ G_{\lambda-{\bf 1}/\mu-{\bf 1}}(\cdot \mid \cdot)$
 using \eqref{eq:Forig}, with the multilinearity
 property on the first column, and \eqref{eq:Flam_minus1} to obtain

\begin{align*}
\lefteqn{(x_d-y_1) G_{\lambda/\mu}({\bf x}\mid {\bf y}) -
  G_{\lambda-{\bf 1}/\mu-{\bf 1}}({\bf x} \mid y_2,\ldots,y_{n-d}) \, = \,
  \frac{(-1)^{d-1}}{D(\lambda)} \, \times } \nonumber \\
&\left( \det\left [ \begin{cases}
      (x_d-y_1), &  j=1\\  (x_i-z_1)\ldots(x_i-z_{\lambda_j+d-j}), &
      j>1 \end{cases} \right]_{i,j=1}^d - \det\left [ \begin{cases}
      (x_i-y_1), &  j=1\\ (x_i-z_1)\ldots(x_i-z_{\lambda_j+d-j}), &
      j>1 \end{cases} \right]_{i,j=1}^d\right) \nonumber \\
&=\frac{(-1)^{d-1}}{D(\lambda)} \det\left [ \begin{cases} x_d-x_i, &
    j=1\\  (x_i-z_1)\ldots(x_i-z_{\lambda_j+d-j}), &  j>1 \end{cases}
                                                    \right]_{i,j=1}^d
                                                     =: (*)\\
\end{align*}

Consider the row $i=d$ in the last determinant. The entries there are
all 0, except when $j=d$: when $j=1$ we have $x_d-x_i=0$ for $i=d$,
when $j\in[2,d-1]$ we have $\lambda_j + d- j \geq \lambda_d+1$,
and since $z_{\lambda_d+1} = x_d$ we have $\prod_{r=1}^{\lambda_j+d-j}
(x_d-z_r) =0$. Using the cofactor expansion we compute the determinant
in the last equation as the principal minor of the matrix times the
$(d,d)$ entry gives
\[
(*) = \frac{(-1)^{d-1}}{D(\lambda)}  \det\left [ \begin{cases} x_d-x_i, &
    j=1\\  (x_i-z_1)\ldots(x_i-z_{\lambda_j+d-j}), &  j>1 \end{cases} \right]_{i,j=1}^{d-1}  (x_d-z_1)\cdots(x_d - z_{\lambda_d}).
\]

We now compare this with equation~\eqref{eq:F_bar}, realizing that
$z_1,\ldots,z_{\lambda_d} = y_1,\ldots,y_{\lambda_d}$, so the last
expression coincides with $G_{\bar{\lambda}/\bar{\mu}}({\bf x} \mid
{\bf y})$ as desired.
Notice also that if $j<i$, we have $\lambda_j +d -j
\geq\lambda_i+d-i+1$, and since $x_i = z_{\lambda_i+d-i+1}$, the terms
above are 0 when $j<i$ and $j \neq 1$.
\end{proof}

\subsection{Proof of NHLF for border strips}

\begin{lemma}\label{lemma:factorial_syt}
Let $\mu \subset \lambda$ be two partitions with at most $d$ parts. Then
\[
\left. \frac{s_{\mu}^{(d)}({\bf x} \mid {\bf
      z}^{\lambda})}{s_{\lambda}^{(d)}({\bf x} \mid {\bf z}^{\lambda})}
\right|_{\substack{x_i=\lambda_i+d-i+1, \\ y_i=d+j-\lambda'_j}} = \frac{f^{\lambda/\mu}}{|\lambda/\mu|!}.
\]
\end{lemma}
An equivalent form of this statement was announced in \cite{Strobl}
(see \cite[\S 8.4]{MPP1}) with a different proof.
\begin{proof}
Let $x_i=\lambda_i + d-i+1$ and $y_j=d+j-\lambda'_j$, then notice that $x_i$ and $y_j$ are exactly the numbers on the horizontal/vertical steps at row $i$/column $j$ of the lattice path determined by $\lambda$ when writing the numbers $1,2,\ldots$ along the path from the bottom left to the top right end. Thus $z_\lambda = 1,2,3\ldots$, and so
$$(x_i-z_1)\cdots (x_i-z_{\mu_j+d-j}) =  (\lambda_i+d-i)\cdots(\lambda_i + d-i +1 - (\mu_j +d-j) )= \frac{ (\lambda_i+d-i)!}{(\lambda_i -i - \mu_j+j)!}$$
whenever $\lambda_i - i \geq \mu_j -j$ and 0 otherwise.  When  $\mu=\lambda$ and $i=j$, we have
$(x_i - z_1)\cdots (x_i - z_{\lambda_i + d-i} ) = (\lambda_i+d-i)!$ and by \eqref{eq:D_upper} we have
$$D(\lambda)\, {\biggl|}_{\substack{x_i=\lambda_i+d-i+1,\\y_j=d+j-\lambda'_j}} \,=\, \prod_{i=1}^d  (\lambda_i+d-i)! $$

Then, by definition and \eqref{eq:GratioD}, we have
\begin{align*}
G_{\lambda/\mu}({\bf x} \mid {\bf y})\Bigg|_{\substack{x_i=\lambda_i+d-i+1,\\y_j=d+j-\lambda'_j}}  \, &\, =\, \frac{D(\mu)}{D(\lambda)}\Bigg|_{\substack{x_i=\lambda_i+d-i+1,\\y_j=d+j-\lambda'_j}} \, = \, \frac{ \det \bigl[   (\lambda_i+d-i)! / (\lambda_i -i - \mu_j+j)! \bigr]_{i,j=1}^d }{\prod_{i=1}^d  (\lambda_i+d-i)!  } \\ &=  \, \det \left[  \frac{1}{(\lambda_i -i - \mu_j+j)!} \right]_{i,j=1}^d
\end{align*}
Multiplying the last determinant by $|\lambda/\mu|!$, we recognize  the exponential specialization of the Jacobi-Trudi identity for the ordinary $s_{\lambda/\mu}$ giving $f^{\lambda/\mu}$ (a formula due to Aitken, see e.g.~\cite[Cor. 7.16.3]{EC2}). Hence we get the desired formula.
\end{proof}

\begin{proof}[Second proof of Theorem~\ref{thm:NHLFborderstrips}]
We start with the relation from Lemma~\ref{lem:key_border_strips} and
evaluate $x_i =\lambda_i+d-i+1$ and $y_j = d+j-\lambda'_j$. In the RHS
by
\eqref{eq:hooksEDborder-strip} we immediately obtain the
RHS of \eqref{eq:NHLFborderstrips}.

Next, we do the same evaluation on the ratio of factorial Schur
functions applying Lemma~\ref{lemma:factorial_syt} that gives the
ratio of factorial Schur functions as $f^{\lambda/\mu}/|\lambda/\mu|!$. \end{proof}

\subsection{SSYT $q$-analogue for border strips} \label{ss:qborderstrip}

To wrap up the section we show how the tools developed to prove
Theorem~\ref{thm:NHLFborderstrips} also yield the SSYT $q$-analogue
for border strips.

\begin{corollary}[\eqref{eq:skewschur} for border strips]\label{cor:qNHLFborderstrips}
For a border strip $\theta=\lambda/\mu$ with end points $(a,b)$ and
$(c,d)$ we have
\begin{equation} \label{eq:qschurborderstrip}
s_{\theta}(1,q,q^2,\ldots,) = \sum_{\substack{\ga:  (a_j,b_j) \to (c_i,d_i),\\ \ga \subseteq \lambda}}
\prod_{(i,j) \in \ga} \frac{q^{\lambda'_j-i}}{1-q^{h(i,j)}}.
\end{equation}
\end{corollary}

\begin{proof}
We start with \eqref{eq:key_border_strips} from Lemma~\ref{lem:key_border_strips} and evaluate both sides at $x_i =q^{\lambda_i+d-i+1}$ and $y_j=q^{d+j-\lambda'_j}$. The path series
$F_{\lambda/\mu}({\bf x} \mid {\bf y})$ gives the  RHS of \eqref{eq:qschurborderstrip}
\begin{equation} \label{eq:pf q-analogue 1}
\left.F_{\lambda/\mu}({\bf x} \mid {\bf y})
\right|_{\substack{x_i=q^{\lambda_i+d-i+1},\\y_j=q^{d+j-\lambda'_j}}}
= (-1)^{|\theta|}\sum_{\substack{\ga:  (a_j,b_j) \to (c_i,d_i),\\ \ga \subseteq \lambda}}
\prod_{(i,j) \in \ga} \frac{q^{-d+\lambda'_j-j}}{1-q^{h(i,j)}}.
\end{equation}
Next, by \cite[\S 4.2]{MPP1} the evaluation of the ratio of the factorial
Schur functions gives the stable principal specialization of the Schur
function, the LHS of \eqref{eq:qschurborderstrip}
\begin{equation} \label{eq:pf q-analogue 2}
\left.\frac{s_{\mu}^{(d)}({\bf x} \mid {\bf
    z}^{\lambda})}{s_{\lambda}^{(d)}({\bf x} \mid {\bf
    z}^{\lambda})}
\right|_{\substack{x_i=q^{\lambda_i+d-i+1},\\y_j=q^{d+j-\lambda'_j}}}
= (-1)^{|\theta|} q^{-g(\lambda)+g(\mu)}s_{\theta}(1,q,q^2,\dots),
\end{equation}
where $g(\nu) := \sum_{i=1}^d \binom{\nu_i+d+1-i}{2}$. Combining
\eqref{eq:pf q-analogue 1} and \eqref{eq:pf q-analogue 2} gives 
\[
s_{\theta}(1,q,q^2,\ldots) = q^{g(\lambda)-g(\mu)} \sum_{\substack{\ga:  (a_j,b_j) \to (c_i,d_i),\\ \ga \subseteq \lambda}}
\prod_{(i,j) \in \ga} \frac{q^{-d+\lambda'_j-j}}{1-q^{h(i,j)}}.
\]
Finally, by a calculation in \cite[Prop. 4.7]{MPP1} the power of $q$
in the RHS above can be rewritten to obtain the RHS of \eqref{eq:qschurborderstrip}.
\end{proof}

\subsection{Lascoux--Pragacz identity for factorial Schur functions} \label{subsec:LP}

Lemma~\ref{lem:key_border_strips} holds for connected skew shape $\lambda/\mu$ in terms of
non-intersecting
paths $\Gamma=(\ga_1,\ldots,\ga_k)$ in $\NIP(\lambda/\mu)$
(i.e. complements of excited diagrams).
\begin{align*}
F_{\lambda/\mu}({\bf x} \mid {\bf y}) &:= \sum_{\Gamma \in
  \NIP(\lambda/\mu)} \prod_{(r,s)\in \Gamma} \frac{1}{x_r-y_s}= \sum_{D\in\ED(\lambda/\mu)} \prod_{(r,s)\in [\lambda]\setminus D} \frac{1}{x_r-y_s}.
\end{align*}
Ikeda and Naruse \cite{IkNa09} showed
algebraically the following identity that we call the {\em multivariate NHLF}.

\begin{theorem}[\cite{IkNa09}] \label{thm:IkNa}
For a connected skew shape $\lambda/\mu \subseteq d\times (n-d)$ we have
\begin{equation} \label{eq:IkNa}
\frac{s_{\mu}^{(d)}({\bf x} \mid {\bf
    z}^{\lambda})}{s_{\lambda}^{(d)}({\bf x} \mid {\bf
    z}^{\lambda})} =
F_{\lambda/\mu}({\bf x} \mid {\bf y}).
\end{equation}
\end{theorem}

In Lemma~\ref{lem:key_border_strips} we proved combinatorially this result
for border strips. We can use the approach from the previous subsections in reverse to obtain a Lascoux--Pragacz
type identity for evaluations of factorial Schur functions.

\begin{corollary} \label{cor:LP4evfactorialSchurs}
If $(\theta_1,\ldots,\theta_k)$ is a Lascoux--Pragacz decomposition of
$\lambda/\mu\subset d\times (n-d)$, then
\begin{equation} \label{eq:LP4evfactorialSchurs}
s_{\mu}^{(d)}({\bf x} \mid {\bf
    z}^{\lambda}) \cdot {s_{\lambda}^{(d)}({\bf x} \mid {\bf
    z}^{\lambda})}^{k-1}
 = \det \left[ s^{(d)}_{\lambda \setminus\, \theta_i
    \# \theta_j}({\bf x} \mid {\bf
    z}^{\lambda}) \right]_{i,j=1}^k
\end{equation}
where $\lambda \setminus\, \theta_i\#\theta_j$ denotes the partition
obtained by removing from $\lambda$ the outer substrip $\theta_i\#\theta_j$.
\end{corollary}

\begin{proof}
By the weighted Lindstr\"om-Gessel-Viennot lemma (Lemma~\ref{lem:LGVlambda}) with
$y_{r,s} = 1/(x_r-y_s)$, we rewrite the RHS of
\eqref{eq:IkNa} as a determinant.
\[
\frac{s_{\mu}^{(d)}({\bf x} \mid {\bf
    z}^{\lambda})}{s_{\lambda}^{(d)}({\bf x} \mid {\bf
    z}^{\lambda})}  = \det\left[\sum_{\substack{\ga:  (a_j,b_j) \to (c_i,d_i),\\ \ga \subseteq \lambda}}
\prod_{(r,s) \in \ga} \frac{1}{x_r-y_s} \right]_{i,j=1}^k= \det \left[
F_{\theta_i\#\theta_j}({\bf x} \mid {\bf y}) \right]_{i,j=1}^k.
\]
Finally, by Lemma~\ref{lem:key_border_strips} each entry of the matrix
can be written as the quotient of $s_{\lambda\setminus\, \theta_i\#\theta_j}^{(d)}({\bf x}\mid
{\bf z_{\lambda}})$ and $s_{\lambda}^{(d)}({\bf x} \mid {\bf
  z_{\lambda}})$. By factoring the denominators out of the matrix we
obtain the result.
\end{proof}

Calculations suggest that an analogue of \eqref{eq:LP4evfactorialSchurs}
holds for general factorial Schur functions $s_{\mu}^{(d)}({\bf x}\mid
{\bf y})$ and not just for the evaluation ${\bf y}={\bf z^{\lambda}}$.

\begin{conjecture} \label{conj:LP4factorialSchurs}
If $(\theta_1,\ldots,\theta_k)$ is a Lascoux--Pragacz decomposition of
$\lambda/\mu\subset d\times (n-d)$, then
\begin{equation} \label{eq:LP4factorialSchurs}
s_{\mu}^{(d)}({\bf x} \mid {\bf y}) \cdot {s_{\lambda}^{(d)}({\bf x}
  \mid {\bf y})}^{k-1}
 = \det \left[ s^{(d)}_{\lambda \setminus\, \theta_i
    \# \theta_j}({\bf x} \mid {\bf
    y}) \right]_{i,j=1}^k
\end{equation}
where $\lambda \setminus\, \theta_i\#\theta_j$ denotes the partition
obtained by removing from $\lambda$ the outer substrip $\theta_i\#\theta_j$.
\end{conjecture}

Since factorial Schur functions reduce to Schur functions when ${\bf
  y}={\bf 0}$, this conjecture implies an identity of Schur functions.

\begin{proposition}
Conjecture~\ref{conj:LP4factorialSchurs} implies the Schur function
identity
\begin{equation} \label{eq:LPratioschurs}
s_{\mu} \cdot s_{\lambda}^{k-1}
 = \det \left[ s_{\lambda \setminus\, \theta_i\# \theta_j} \right]_{i,j=1}^k,
\end{equation}
where $(\theta_1,\ldots,\theta_k)$ is a Lascoux--Pragacz decomposition of
$\lambda/\mu$.
\end{proposition}

\begin{example}
For the Lascoux--Pragacz decomposition in Example~\ref{ex:LP},
\eqref{eq:LPratioschurs} gives the following identity that indeed
holds:
\[
s_{(2,1)}s_{(5,4^2,1)} = s_{(3^2)}s_{(5,3,2,1)}-s_{(3^2,2,1)}s_{(5,3)}.
\]
\end{example}

\begin{remark}
Note that instead of reversing the approach in
Section~\ref{sec:all_shapes}, having a combinatorial proof of the identity in Corollary~\ref{cor:LP4evfactorialSchurs} would
show that the multivariate NHLF (Theorem~\ref{thm:IkNa}) for skew shapes is
equivalent to the multivariate NHLF for border strips
(Lemma~\ref{lem:key_border_strips}).
\end{remark}

\bigskip

\section{Excited diagrams and SSYT of border strips and thick strips}  \label{sec:enum_strips_SSYT}

\nin
In the next two sections we focus on the case of the  {\em thick strip}
$\delta_{n+2k}/\delta_n$ where $\delta_n$ denotes the staircase shape
$(n-1,n-2,\ldots,2,1)$. We study the excited diagrams
$\ED(\delta_{n+2k}/\delta_n)$ using the results from Section~\ref{sec:excited_paths} and the number of SYT of this shape combining
the NHLF, its SSYT $q$-analogue (Theorem~\ref{thm:skewSSYT}) and the
Lascoux--Pragacz identity.

\smallskip

\subsection{Excited diagrams and Catalan numbers} \label{ss:excited_fandyck}
%
%

We start by enumerating the excited diagrams of the shape $\delta_{n+2k}/\delta_n$.

\begin{corollary} \label{cor:excitedcat}
We have: \. $\sED(\delta_{n+2}/\delta_n) \ts = \ts C_{n}$, \,
\. $\sED(\delta_{n+4}/\delta_n) \ts = \ts C_n\ts C_{n+2} -
C_{n+1}^2$, and in general
\begin{equation} \label{eq:excitedcat}
\sED(\delta_{n+2k}/\delta_n) \. =  \, \det [C_{n-2+i+j}]^{k}_{i,j=1}
\.=\, \prod_{1\leq i < j \leq n} \frac{2k+i+j-1}{i+j-1}\..
\end{equation}
\end{corollary}

\begin{proof}
We start with the case $k=1$ for the zigzag border strip $\delta_{n+2}/\delta_n$. By
Proposition~\ref{prop:NIP2ED} the complement of excited diagrams of
$\delta_{n+2}/\delta_n$ are paths $\ga:(n+1,1)\to (1,n+1)$,
$\ga\subseteq \delta_{n+2}$. By rotating these paths $45^{\circ}$
clockwise one obtain the Dyck paths in $\Dyck(n)$ as illustrated in
Figure~\ref{fig:excited2cat}. Thus $\sED(\delta_{n+2}/\delta_n)=C_n$.

For general $k$, the shape $\delta_{n+2k}/\delta_n$ has a
Lascoux--Pragacz decomposition into $k$ maximal border strips
$(\theta_1,\ldots,\theta_k)$ where $\theta_i$ is the zigzag strip from
$(n+2k-2i-1,1)$ to $(1,n+2k-2i-1)$ (see
Figure~\ref{fig:excited_cat_hooks}: Left). Then by
Theorem~\ref{thm:num_excited_HG} we have
\[
\sED(\delta_{n+2k}/\delta_n) =  \det {\big [}\, \sED({\theta_i \# \theta_j})\, {\big ]}_{i,j=1}^k.
\]
The cutting strip $\tau$ of the decomposition of $\delta_{n+2k}/\delta_n$ is
the zigzag $\theta_1$. The strips $\theta_i \# \theta_j$ in the
determinant, being substrips of $\theta_1$, are themselves zigzags.
The strip $\theta_i \# \theta_j$ in $\theta_1$ consists of the
cells with content from $2+2j-n-2k$ to $n+2k-2i-2$. So the strip is a
zigzag $\delta_{m+2}/\delta_m$ of size $2m+1$ where $m=n+2k+i+j+2$. Since we
already know that the shape $\delta_{m+2}/\delta_m$ has $C_m$ excited
diagrams then the above determinant becomes
\[
\sED(\delta_{n+2k}/\delta_n) = \det {\big [}\, C_{n+2k-i-j-2}\, {\big
                               ]}_{i,j=1}^k=\det {\big [}\,
                               C_{n+i+j-2}\, {\big ]}_{i,j=1}^k,
\]
where the last equality is obtained by relabeling the matrix. This
proves the first equality of \eqref{eq:excitedcat}.

To prove the second equality we use the characterization of excited
diagrams as flagged tableaux. By Proposition~\ref{prop:excited2flagged}, excited diagrams
in $\ED(\delta_{n+2k}/\delta_n)$ are in bijection with flagged
tableaux of shape $\delta_n$ with flag $(k+1,k+2,\ldots,k+n-1)$. By
subtracting $i$ to all entries in row $i$, these tableaux are
equivalent to reverse plane partitions of shape $\delta_n$ with
entries $\leq k$ which are counted by the given product
formula due to Proctor \cite{Pr2}.
\end{proof}

\begin{remark}
Similarly excited
diagrams in $\ED(\delta_{n+2k+1}/\delta_n)$ are in
correspondence with flagged tableaux of shape $\delta_n$ with flag
$(k+1,k+2,\ldots,k+n-1)$, thus
$|\ED(\delta_{n+2k}/\delta_n)|  = |\ED(\delta_{n+2k+1}/\delta_n)|$. In
what follows the formulas for the even case $\delta_{n+2k}$ are
simpler than those of the odd case so we omit the latter.
\end{remark}

\begin{figure}[h]
\begin{center}
\includegraphics[scale=0.8]{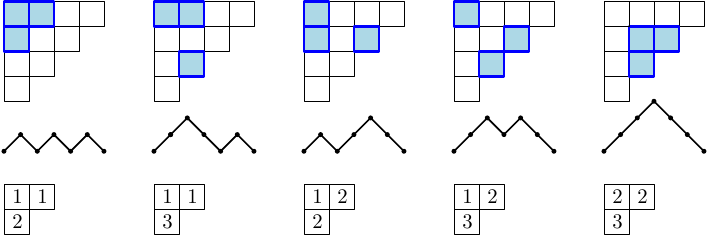}
\caption{Correspondence between excited diagrams in
  $\delta_5/\delta_3$, Dyck paths in $\Dyck(3)$ and flagged tableaux of shape $\delta_3$ with flag
  $(2,3)$.}
\label{fig:excited2cat}
\end{center}
\end{figure}

From the first determinantal formula for $\sED(\lambda/\mu)$
(Proposition~\ref{prop:GV}) we easily obtain the following curious
determinantal identity (see also~$\S$\ref{ss:finrem-cat}).

\begin{corollary}  \label{cor:cat-det}
We have:
\[
 \det\left[ \binom{n-i+j}{i} \right]_{i,j=1}^{n-1} \, = \, C_n\..
\]
\end{corollary}

\begin{proof}
By Corollary~\ref{cor:excitedcat}, we have $|\ED(\delta_{n+2}/\delta_n)|=C_n$.
We apply Proposition~\ref{prop:GV} to the shape $\delta_{n+2}/\delta_n$,
where the vector $\f^{\delta_{n+2}/\delta_n} = (2,3,\ldots,n)$, see~$\S$\ref{ss:excited-flagged}.
This expresses \ts $|\ED(\delta_{n+2}/\delta_n)|$ \ts as the
given determinant, and the identity follows.
\end{proof}

Next we give a description of the excited diagrams of the shape
$\delta_{n+2k}/\delta_n$. Let $\FDyck(k,n)$ be the set of tuples  $(p_1,\ldots,p_k)$ of $k$ noncrossing Dyck paths from $(0,0)$
to $(2n,0)$ (see Figure~\ref{fig:excited_cat_hooks}: Right). We call
such tuples {\em $k$-fans of Dyck paths}. It is known \cite{StCath-Viennot} that fans of
Dyck paths are counted by the determinant
of Catalan numbers and the product formula in \eqref{eq:excitedcat}.

\begin{corollary} \label{cor:excited2fandyckpaths}
We have $\sED(\delta_{n+2k}/\delta_n) =|\FDyck(k,n)|$ and the
complements of the excited diagrams correspond to $k$-fans of paths in $\FDyck(k,n)$.
\end{corollary}

\begin{proof}
By Proposition~\ref{prop:NIP2ED} the complements of excited diagrams
in $\ED(\delta_{n+2k}/\delta_n)$ correspond to $k$-tuples of
nonintersecting paths in $\NIP(\delta_{n+2k}/\delta_n)$ (paths
obtained via ladder moves from the original paths
$(\ga^*_1,\ldots,\ga^*_k)$ of the Kreiman outer decomposition of
$\delta_{n+2k}/\delta_n$).

The path $\ga^*_i$ consists of zigzag path $p^*_i$ of
$2n+1$ cells bookended by a vertical and horizontal segment of $k-i$
cells each (see Figure~\ref{fig:excited_cat_hooks}:Middle). Because the excited/ladder moves preserve the
contents of the cells of $\delta_n$, the path $\ga_i$ in  $(\ga_1,\ldots,\ga_k)
\in \NIP(\delta_{n+2k}/\delta_n)$ will consist of a Dyck path $p_i$
bookended by the same vertical and
horizontal segments as in $\ga^*_i$. Thus the map $(\ga_1,\ldots,\ga_k)
\mapsto (p_1,\ldots,p_k)$ denoted by $\varphi$ is a correspondence between
$\NIP(\delta_{n+2k}/\delta_n)$ and $\FDyck(k,n)$. See
Figure~\ref{fig:excited_cat_hooks}, right, for an example.
\end{proof}

\begin{remark}
Fans of Dyck paths in $\FDyck(k,n)$ are equinumerous with
{\em $k$-triangulations} of an $(n+2k)$-gon \cite{Jonsson}
(see also~\cite[A12]{StCat} and \cite{SS} for a bijection for general $k$).
\end{remark}

\begin{center}
\begin{figure} [hbt]
\includegraphics{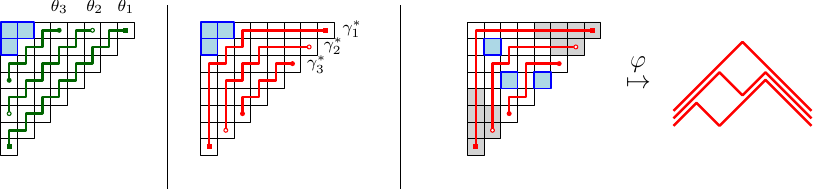}
\caption{Left, Middle: the Lascoux--Pragacz and the Kreiman outer
  decompositions of the shape $\delta_{3+6}/\delta_3$. Right: the hook-lengths of an excited diagram of
  $\delta_{3+6}/\delta_3$ corresponding to the $3$-fan of Dyck paths on
  the right. Each gray area has cells with product of hook-lengths $\ts (3!! \cdot 7!!)$.}
\label{fig:excited_cat_hooks}
\end{figure}
\end{center}

\medskip \subsection{Determinantal identity of Schur functions of thick strips}
Observe that SYT of shape $\delta_{n+2}/\delta_n$ are in bijection
with {\em alternating permutations} of size $2n+1$. These permutations
are counted by the odd Euler number $E_{2n+1}$. Thus,
\[
f^{\delta_{n+2}/\delta_n} \. = \. E_{2n+1}\..
\]
Let $E_n(q)$ be as in the introduction, the $q$-analogue
of Euler numbers.\footnote{In the survey
\cite[\S 2]{Stanley_SurveyAP}, our $E_n(q)$ is denoted by
$E^{\star}_n(q)$.}

\begin{example} \label{ex:qEulerA}
We have: \. $E_1(q)=E_2(q) =1$, \. $E_3(q) \ts = \ts q^2+q$, \. $E_4(q)  \ts = \ts q^4+q^3+ 2q^2+q$, \ts and

\nin
$E_5(q)\ts = \ts q^8+2q^7+3q^6+4q^5+3q^4+2q^3+q^2$\ts.
\end{example}

By the theory of $(P,\omega)$-partitions
\cite[$\S$2]{Stanley_SurveyAP}\cite[$\S$7.19]{EC2},  we have:
\begin{equation}
E_{2n+1}(q) \. = \. s_{\delta_{n+2}/\delta_n}(1,q,q^2,\ldots)\cdot \prod_{i=1}^{2n+1} (1-q^i)\..
 \label{eq:qEuler}
\end{equation}

Next we apply the Lascoux--Pragacz identity to the shape
$\delta_{n+2k}/\delta_n$.

\begin{corollary}[Lascoux--Pragacz for $\delta_{n+2k}/\delta_n$]
\label{cor:schurzigzagdet}
We have:
$$
s_{\delta_{n+2k}/\delta_n}({\bf x}) \, = \, \det \left[
  s_{\delta_{n+i+j}/\delta_{n-2+i+j}}({\bf x})\right]^{k}_{i,j=1}\,.
$$
\end{corollary}

\begin{proof}
By Theorem~\ref{thm:LascouxPragacz} for the shape
$\delta_{n+2k}/\delta_n$ we have
\[
s_{\delta_{n+2k}/\delta_n}({\bf x}) = \det{\big [} \, s_{\theta_i \#
  \theta_j}({\bf x})\, {\big ]}_{i,j=1}^k,
\]
where $(\theta_1,\ldots,\theta_k)$ is the decomposition of the shape
$\delta_{n+2k}/\delta_n$ into $k$ maximal border strips. As in the
proof of Corollary~\ref{cor:excitedcat}, the strip
$\theta_i\#\theta_j$ has zigzag shape $\delta_{m+2}/\delta_m$ for
$m=n+2k-i-j+2$. Thus, after relabeling the matrix, the above equation
becomes the desired expression.
\end{proof}

\begin{corollary}
\label{cor:qEdet}  We have:
$$
s_{\delta_{n+2k}/\delta_n}(1,q,q^2,\ldots) \, = \,
\det\left[\widetilde{E}_{2(n+i+j)-3}(q)\right]^{k}_{i,j=1}\,,
$$
where
$$
\widetilde{E}_n(q) \, := \, \frac{E_n(q)}{(1-q)(1-q^2)\cdots (1-q^n)}\,.
$$
\end{corollary}

\begin{proof}
The result follows from Corollary~\ref{cor:schurzigzagdet} and equation~\eqref{eq:qEuler}.
\end{proof}

Taking the limit $q\to 1$ in Corollary~\ref{cor:qEdet} we get corresponding
identities for $f^{\delta_{n+2k}/\delta_n}$.
\begin{corollary} \label{cor:Edet}
We have:
$$
\frac{f^{\delta_{n+2k}/\delta_n}}{|\delta_{n+2k}/\delta_n|!} \, = \, \det \left[ \wh{E}_{2(n+i+j)-3}\right]^{k}_{i,j=1}\,,
\quad \text{where} \ \  \wh{E}_n \. :=\, \frac{E_n}{n!}\,.
$$
\end{corollary}

\begin{remark}
Baryshnikov and Romik~\cite{BR} gave similar determinantal formulas
for the number of standard Young tableaux of skew shape
$(n+m-1,n+m-2,\ldots,m)/(n-1,n-2,\ldots,1)$, extending the method of
Elkies (see e.g.~\cite[Ch.~14]{HEC}).
\end{remark}

In a different direction, one can use Corollary~\ref{cor:Edet} when $n=1,2$ to
obtain the following determinant formulas for Euler numbers in terms of $f^{\delta_{2k+1}}$ and $f^{\delta_{2k}}$, which of course can be computed
by a~HLF (cf.~\cite[\href{https://oeis.org/A005118}{A005118}]{OEIS}).

\begin{corollary} We have:
$$
\det\Bigl[\wh{E}_{2(i+j)-1}\Bigr]_{i,j=1}^k \, = \, \frac{f^{\delta_{2k+1}}}{\binom{2k+1}{2}!}\,, \qquad
\det\Bigl[\wh{E}_{2(i+j)+1}\Bigr]_{i,j=1}^k \, = \, \frac{f^{\delta_{2k}}}{\bigl(\binom{2k}{2}-1\bigr)!}\,.
$$
\end{corollary}

\medskip \subsection{SYT and Euler numbers}
We use the NHLF to obtain an expression
for $f^{\delta_{n+2}/\delta_n}=E_{2n+1}$ in terms of Dyck paths.
%

\begin{proof}[Proof of Corollary~\ref{cor:euler-nhlf}]
By the NHLF, we have
\begin{equation} \label{eq:pfcat2euler}
f^{\delta_{n+2}/\delta_n} \, = \, |\delta_{n+2}/\delta_n|! \. \sum_{D \in
  \ED(\delta_{n+2}/\delta_n)} \. \prod_{u\in \overline{D}} \. \frac{1}{h(u)}\,,
\end{equation}
where $\overline{D}=[\delta_{n+2}/\delta_n ]\setminus D$. Now
$|\delta_{n+2}/\delta_n|=(2n+1)!$ and by
 Corollary~\ref{cor:excitedcat} (complements of) excited
diagrams $D$ of $\delta_{n+2}/\delta_n$ correspond to Dyck paths $\gamma$ in
$\Dyck(n)$. In this correspondence, if $u\in
\overline{D}$ corresponds to point $(a,b)$ in $\gamma$ then
$h(u)=2b+1$ (see Figure~\ref{fig:excited2cat}). Translating from excited diagrams
 to Dyck paths, \eqref{eq:pfcat2euler} becomes the desired identity~\eqref{eq:cat2euler}.
\end{proof}

Equation~\eqref{eq:cat2euler} can be generalized to thick strips
$\delta_{n+2k}/\delta_n$.

\begin{corollary}\label{cor:kEulerPath}
\label{eq:cat4euler}  We have:
\begin{equation} \label{eq:cat2keuler}
\sum_{\substack{(\p_1,\ldots,\p_k) \in
    \Dyck(n)^k\\ \text{\rm{noncrossing}}}} \, \,
\prod_{r=1}^k \. \prod_{(a,b) \in
  \p_r}\frac{1}{2b + 4r-3} \,\, = \,\,
\left[\prod_{r=1}^{k-1} \. (4r-1)!!\right]^2 \.
\det \left[   \wh{E}_{2(n+i+j)-3}\right]^{k}_{i,j=1}\,,
\end{equation}
where $\wh{E}_n = E_n/n!$ and $(a,b)\in \p$ denotes a point of the Dyck path $\p$.
\end{corollary}

\begin{proof}
For the RHS we use Corollary~\ref{cor:Edet} to express
$f^{\delta_{n+2k}/\delta_n}$ in terms of Euler numbers. For the LHS, we first
use the NHLF to write $f^{\delta_{n+2k}/\delta_n}$ as a sum
over excited diagrams $\ED(\delta_{n+2k}/\delta_n)$~:
\[
f^{\delta_{n+2k}/\delta_n} \, = \, |\delta_{n+2k}/\delta_n|! \. \sum_{D \in
  \ED(\delta_{n+2k}/\delta_n)} \. \prod_{u\in \overline{D}} \. \frac{1}{h(u)}\.,
\]
where $\overline{D}=[\delta_{n+2k}/\delta_n ]\setminus D$. By Corollary~\ref{cor:excited2fandyckpaths}, excited
diagrams of $\delta_{n+2k}/\delta_n$ correspond to $k$-tuples of
noncrossing Dyck paths in $\FDyck(k,n)$ via the map $\varphi$. Finally, one can check (see
Figure~\ref{fig:excited_cat_hooks} right) that if
$\varphi:D\mapsto (\p_1,\ldots,\p_k)$ then
\[
\prod_{u\in \overline{D}} h(u) \, = \, \bigg [\prod_{r=1}^{k-1}
  (4r-1)!! {\bigg ]^2
  \prod_{(a,b)\in \p_r} (2b+4r-3)  }\,,
\]
which gives the desired RHS.
\end{proof}

\medskip \subsection{Probabilistic variant of
  \eqref{eq:cat2euler}}
Here we present a new identity~\eqref{eq:cat2eulerB} which is a close
relative of the curious identity~\eqref{eq:cat2euler} we proved above.

Let $\BT(n)$ be the set of {\em plane full binary trees} $\tau$
with $2n+1$ vertices, i.e. plane binary trees where every vertex has
zero or two descendants. These trees are counted by $|\BT(n)|=C_n$ (see
e.g. \cite[\S 2]{StCat}). Given a
vertex $v$ in a tree $\tau \in \BT(n)$, $h(v)$ denotes the number of
descendants of $v$ (including itself). An {\em increasing}  labeling
of $\tau$ is a labeling $\omega(\cdot)$ of the vertices of $\tau$ with
$\{1,2,\ldots,2n+1\}$ such that if $u$ is a descendant of $v$ then
$\omega(v)\leq \omega(u)$. By abuse of notation, let $f^{\tau}$ is the
number of increasing labelings of $\tau$. By the HLF for trees
(see e.g.~\cite{BStree}), we have:
\begin{equation} \label{eq:incr-trees}
f^{\tau} \. = \. \frac{(2n+1)!}{\prod_{v\in \tau} h(v)}\..
\end{equation}

\begin{proposition}\label{prop:cat2euler}
We have:
\begin{equation} \label{eq:cat2eulerB}
\sum_{\tau \in \BT(n)} \. \prod_{v\in \tau} \. \frac{1}{h(v)} \, = \, \frac{E_{2n+1}}{(2n+1)!}\,.
\end{equation}
\end{proposition}

\begin{proof} 
The RHS of~\eqref{eq:cat2eulerB} gives the probability
 $E_{2n+1}/(2n+1)!$ that a permutation $w\in \SS_{2n+1}$ is alternating. We use
the representation of a permutation $w$ as an {\em increasing binary tree}
$T(w)$ with $2n+1$ vertices (see e.g.~\cite[\S
1.5]{EC2}). It is well-known
that $w$ is an {\em down-up} permutation (equinumerous with
up-down/alternating permutations) if and only if $T(w)$ is an increasing
full binary tree \cite[Prop. 1.5.3]{EC2}. See Figure~\ref{fig:alt2fullbintree} for an
example.  We conclude that the
probability $p$ that an increasing binary tree is a full binary tree
is given by $p=E_{2n+1}/(2n+1)!$.

On the other hand, we have:
\[
p \. = \, \sum_{\tau\in \BT(n)} \frac{f^{\tau}}{(2n+1)!}\,,
\]
where $f^{\tau}/(2n+1)!$ is the probability that a labeling of a
full binary tree $\tau$ is increasing. By~\eqref{eq:incr-trees}, the
result follows.
\end{proof}

\begin{figure}[hbt]
\begin{center}
\includegraphics{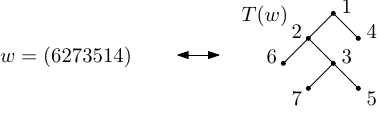}
\caption{The full binary tree corresponding to the alternating permutation $w=(6273514)$.}
\label{fig:alt2fullbintree}
\end{center}
\end{figure}

\begin{remark}
Note the similarities between~\eqref{eq:cat2eulerB} and~\eqref{eq:cat2euler}.
They have the same~RHS,
both are sums over the same number~$C_n$ of Catalan objects of products
of $n$ terms, and both are variations on the (usual) \eqref{eq:hlf} for other
posets.  As the next example shows, these equations are quite different.
\end{remark}

\begin{example}
For $n=2$ there are $C_2=2$  full binary trees with $5$ vertices and
$E_5=16$. By Equation~\eqref{eq:cat2eulerB}
\[
\frac{1}{3\cdot 5} + \frac{1}{3\cdot 5} \. = \. \frac{16}{5!}\..
\]
On the other hand, for the two Dyck paths in $\Dyck(2)$,  Equation~\eqref{eq:cat2euler} gives
\[
\frac{1}{3\cdot 3} + \frac{1}{3\cdot 3 \cdot 5}  \. = \. \frac{16}{5!}\..
\]
\end{example}

\medskip \subsection{Identity
  \eqref{eq:cat2euler} for other types}

In this section $\lambda$ and $\mu$ are partitions with distinct
parts. We consider shifted diagrams of shape $\lambda$ and skew shape
$\lambda/\mu$ and standard tableaux of shifted shape $\lambda/\mu$. Along with Theorem~\ref{thm:IN}, Naruse also announced two formulas for the number
$g^{\lambda/\mu}$ of standard tableaux of skew shifted shape~$\lambda/\mu$,
in terms of type $B$ and type $D$ excited diagrams. These
excited diagrams are obtained from the diagram of $\mu$ by applying
the following excited moves:

\begin{center}
\begin{tabular}{lccc}
type $B$:& \includegraphics{excited_move} & and &
\includegraphics{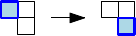}\\
type
$D$: &  \includegraphics{excited_move} & and & \includegraphics{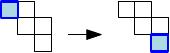}.
\end{tabular}
\end{center}

We denote the set of type $B$ (type $D$) excited diagrams of shifted skew shape
$\lambda/\mu$ by $\ED^B(\lambda/\mu)$  ($\ED^D(\lambda/\mu)$). As in
Section~\ref{ss:excited-flagged} or \cite[\S 3]{MPP1}, type $B$ (type
$D$) excited diagrams of $\lambda/\mu$ are equivalent to
certain (subset of) flagged tableaux of shifted shape $\mu$ and to certain
non-intersecting paths (see Figure~\ref{fig:Bexcited2cat}).

Given a shifted shape $\lambda$, the type $B$ hook of a cell $(i,i)$ on the diagonal is the cells in
row $i$ of $\lambda$. The hook of a cell $(i,j)$ for $i\leq j$ is the cells in row $i$ right of $(i,j)$, the
cells in column $j$ below $(i,j)$, and if $(j,j)$ is one of these cells below
then the hook also includes the cells in the $j$th row of $\lambda$
(overall counting $(j,j)$ twice). The type $D$ hook is the usual
shifted hook (e.g., see \cite[Ex. 3.21]{SSG}) The hook-length of $(i,j)$ is
the size of the hook of $(i,j)$ and is denoted by $h^B(i,j)$
($h^D(i,j)$); see Figure~\ref{fig:BDhooks}.

\begin{figure}
\includegraphics{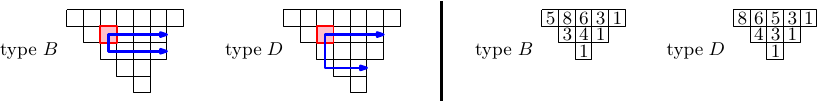}
\caption{Left: examples of the type $B$ and type $D$ hook of a cell $(i,j)$ of $\lambda$
  of lengths $9$ (cell $(3,3)$ is counted twice) and $7$ respectively. Right: the type $B$ and $D$
  hook-lengths of the cells of the shifted shape $(5,3,1)$.}
\label{fig:BDhooks}
\end{figure}

 The NHLF then extends verbatim.

\begin{theorem}[Naruse \cite{Strobl}]
Let $\lambda, \mu$ be partitions with distinct parts, such that $\mu
\subset \lambda$. We have
\begin{align}
g^{\lambda/\mu} &= |\lambda/\mu|! \sum_{S \in
  \ED^B(\lambda/\mu)} \prod_{(i,j) \in [\lambda]\setminus S}
                  \frac{1}{h^B(i,j)}, \label{eq:nhlfB}\\
&=|\lambda/\mu|! \sum_{S \in
  \ED^D(\lambda/\mu)} \prod_{(i,j) \in [\lambda]\setminus S}
  \frac{1}{h^D(i,j)} \label{eq:nhlfD},
\end{align}
where $h^B(i,j)$ and $h^D(i,j)$ are the shifted hook-lengths of type~B
and type~D, respectively.
\end{theorem}

\begin{example}[shifted thick zigzag strip]
The shifted  analogue of the staircase is the \emph{triangle} $\nabla_n =
(2n-1,2n-3,\ldots,1)$. The analogue of the thick strip is the shifted skew
shape $\nabla_{n+k}/\nabla_n$. The number of type
$B$ excited diagrams of this shape has a product formula analogous to
\eqref{eq:excitedcat}.

\begin{proposition}
\begin{equation} \label{eq:typeBzigzag}
|\ED^B(\nabla_{n+k}/\nabla_n)| = \prod_{h=1}^k \prod_{i=1}^n \prod_{j=1}^n \frac{h+i+j-1}{h+i+j-2}.
\end{equation}
\end{proposition}

\begin{proof}
As in the standard shape case, the type $B$ excited diagrams
correspond to shifted flagged tableaux of shape $\nabla_n$ with
entries in row $i$  $\leq i+k$. By subtracting $i$ from all entries in
row $i$ of such tableaux they are equivalent to plane partitions of shape $\nabla_n$
with entries $\leq k$. By a result of Proctor \cite{Pr1}, recently
proved bijectively in \cite{HPPW}, these are equinumerous with plane
partitions in a $n\times n \times k$ box (see also \cite{HW}). Thus
the result follows by
MacMahon's boxed plane partition formula.
\end{proof}

In the case $k=1$ we obtain $|\ED^B(\nabla_{n+1}/\nabla_n)| =
\binom{2n}{n}$ (see Figure~\ref{fig:Bexcited2cat}). When
$k=n$,  $|\ED^B(\nabla_{2n}/\nabla_n)|$ counts plane partitions that fit inside the $n\times n\times n$ box
(see e.g. \cite[\href{http://oeis.org/A008793}{A08793}]{OEIS}).
\end{example}

The shape $\nabla_{n+1}/\nabla_n$ is a zigzag and so $g^{\nabla_{n+1}/\nabla_n} =
E_{2n+1}$ (see Figure~\ref{fig:Bexcited2cat}). Thus, as a corollary of
\eqref{eq:nhlfB}, we obtain a type $B$ variant
of the Euler-Catalan identity \eqref{eq:cat2euler}. Let $\Dyck^B(n)$
be the set of lattice paths $\mathsf{p}$ starting at $(0,0)$ with steps $(1,1)$ and
$(1,-1)$ of length $2n$ that stay on or above the $x$-axis. Note that
$|\Dyck^B(n)| = \binom{2n}{n}$, sometimes called the type $B$ Catalan
number \cite[\href{https://oeis.org/A000984}{A000984}]{OEIS}.

\begin{corollary}
\begin{equation} \label{eq:ECtypeB}
\sum_{\mathsf{p} \in \Dyck^B(n)} \prod_{(a,b)\in \mathsf{p}} \frac{1}{wt(a,b)}
  = \frac{E_{2n+1}}{(2n+1)!}, \quad \text{where} \quad wt(a,b) = \begin{cases}
2b+1 & \text{ if } a \leq n,\\
2b+2 & \text{ if }  n<a < 2n,\\
b+1 & \text{ if } a=2n.
\end{cases}
 \tag{EC-B}
\end{equation}
\end{corollary}

\begin{example}
Figure~\ref{fig:Bexcited2cat} shows the $\binom{4}{2}$ excited
diagrams of shape $\nabla_3/\nabla_2$. By taking their complements and
reflecting vertically, we obtain the paths in $\Dyck^B(2)$. Either using
$wt(a,b)$ on the paths or the
hook-lengths for the shape $\nabla_3/\nabla_2$ (see
Figure~\ref{fig:BDhooks} right), \eqref{eq:ECtypeB} gives
\[
\frac{1}{4\cdot 3\cdot 1^3} \. + \. \frac{1}{6\cdot 4 \cdot 3 \cdot 1^2}  \. + \.
\frac{1}{4\cdot 3^2\cdot 1^2}  \. + \.  \frac{1}{6\cdot 4\cdot 3^2\cdot 1}  \. + \.
\frac{1}{8\cdot 6 \cdot 3^2\cdot 1}  \. + \.  \frac{1}{8\cdot 6 \cdot 5 \cdot
  3 \cdot 1} \, = \, \frac{16}{5!}\..
\]
\end{example}

\begin{figure}
\includegraphics[scale=0.8]{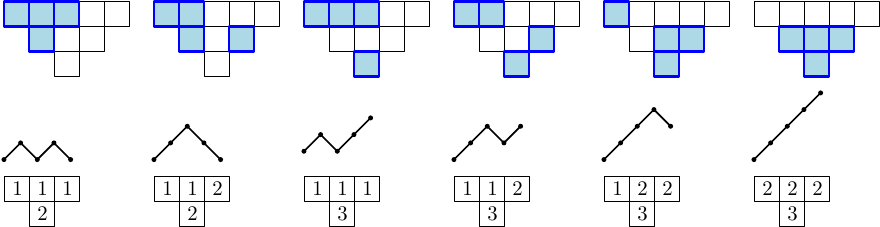}
\caption{Correspondence between type $B$ excited diagrams in
  $\nabla_3/\nabla_2$, paths in $\Dyck^B(2)$ and flagged tableaux of shape $\nabla_2$ with flag
  $(2,3)$.}
\label{fig:Bexcited2cat}
\end{figure}

\begin{remark}
The complements of type $D$ excited diagrams of the shape
  $\nabla_{n+1}/\nabla_n$ are just the Dyck paths in $\Dyck(n)$, thus
  $|\ED^D(\nabla_{n+1}/\nabla_n)| = C_n$. In addition, one can see that \eqref{eq:nhlfD} for $\nabla_{n+1}/\nabla_n$ is just \eqref{eq:cat2euler}. It would be of interest to
  find a formula for $|\ED^D(\nabla_{n+k}/\nabla_n)|$ analogous to
  \eqref{eq:typeBzigzag}.
\end{remark}

\medskip \subsection{$q$-analogue of Euler numbers via SSYT}
We use our first $q$-analogue of NHLF (Theorem~\ref{thm:skewSSYT})  to obtain identities
for $s_{\delta_{n+2k}/\delta_n}(1,q,q^2,\ldots)$ in terms of Dyck
paths.



\begin{proof}[Proof of Corollary~\ref{cor:euler-nhlf-ssyt}]
By Theorem~\ref{thm:skewSSYT} for the skew shape $\delta_{n+2}/\delta_n$ and \eqref{eq:qEuler} we have
\begin{equation} \label{eq:pf1stqEuler}
\frac{E_{2n+1}(q)}{(1-q)(1-q^2)\cdots (1-q^{2n+1}) }= \sum_{D \in
  \ED({\delta_{n+2}/\delta_n})} \prod_{(i,j)\in
  [\delta_{n+2}]\setminus D} \frac{q^{\lambda'_j-i}}{1-q^{h(i,j)}}.
\end{equation}
Suppose $D$ in $\ED(\delta_{n+2}/\delta_n)$ corresponds to the Dyck path
$\p$ and cell $(i,j)$ in $D$ corresponds to point
$(a,b)$ in $\p$ then $h(i,j)=2b+1$ and $\lambda'_j-i=b$. Using
this correspondence, the LHS of \eqref{eq:pf1stqEuler} becomes the LHS
of the
desired expression.
\end{proof}

\begin{corollary}\label{cor:thick-det-ssyt}
$$
\sum_{\substack{(\p_1,\ldots,\p_k) \in
    \Dyck(n)^k\\ \text{\rm{noncrossing}}}} \, \. \prod_{r=1}^k \. \prod_{(a,b) \in
  \p_r}\frac{q^{b+2r-2}}{1-q^{2b + 4r-3}} \.
  \, = \,  \.
\left(\prod_{r=1}^{k-1} \. [4r-1]!!\right)^2 \,
\det \left[ \wt{E}_{2(n+i+j)-3}(q)\right]^{k}_{i,j=1}\,
 $$
where \. $\widetilde{E}_n(q) := E_n(q)/(1-q)(1-q^2)\cdots (1-q^n)$ \. and \.
$[2m-1]!!:=(1-q)(1-q^3)\cdots (1-q^{2m-1})$\ts.
\end{corollary}

\begin{proof}
For the LHS, use Corollary~\ref{cor:qEdet} to express
$s_{\delta_{n+2k}/\delta_n}(1,q,q^2,\ldots)$ in terms of $q$-Euler polynomials
$\wt E_m(q)$. For the RHS, first use Theorem~\ref{thm:skewSSYT} for the skew
shape $\delta_{n+2k}/\delta_n$ and then follow the same argument as
that of Corollary~\ref{cor:kEulerPath}.
\end{proof}

\begin{remark}
Combinatorial proofs of Corollary~\ref{cor:euler-nhlf} and
Corollary~\ref{cor:euler-nhlf-ssyt} are
obtained from the proof in Section~\ref{sec:all_shapes} of
\eqref{eq:Naruse}. Similarly, combinatorial proofs of Corollary~\ref{cor:kEulerPath} and Corollary~\ref{cor:thick-det-ssyt}
are obtained from the proof in Section~\ref{sec:all_shapes} of
\eqref{eq:Naruse} and the first $q$-NHLF for all shapes,
or from Konvalinka's bijective proof of \eqref{eq:Naruse} in \cite{Ko}
and the first $q$-NHLF for border strips.\footnote{Jang Soo Kim has a direct proof of Corollary~\ref{cor:euler-nhlf-ssyt}
using continued fractions and orthogonal polynomials (private
communication).}
\end{remark}

\bigskip\section{Pleasant diagrams and RPP of border strips and thick strips} \label{sec:enum_strips_RPP}

\nin
In this section we study pleasant diagrams in
$\PD(\delta_{n+2}/\delta_n)$ and our second $q$-analogue of NHLF
(Theorem~\ref{thm:skewRPP}) for RPP of shape
$\delta_{n+2}/\delta_n$. Recall that $\sPD(\lambda/\mu)$ denotes the
number of pleasant diagrams of shape $\lambda/\mu$.

\subsection{Pleasant diagrams and Schr\"{o}der numbers} \label{sec:pleas_is_schroeder}
Let $s_n$ be the $n$-th {\em  Schr\"oder number}
\cite[\href{https://oeis.org/A001003}{A001003}]{OEIS} which counts
lattice paths from $(0,0)$ to $(2n,0)$ with steps $(1,1)$, $(1,-1)$,
and $(2,0)$ that never
go below the $x$-axis and no steps $(2,0)$ on the $x$-axis.

\begin{theorem} \label{prop:pleas-schr}
We have: \ts
$\sPD(\delta_{n+2}/\delta_n)=2^{n+2}s_n$\ts, for all $n \ge 1$.
\end{theorem}

\begin{figure}[hbt]
\begin{center}
\includegraphics[scale=0.8]{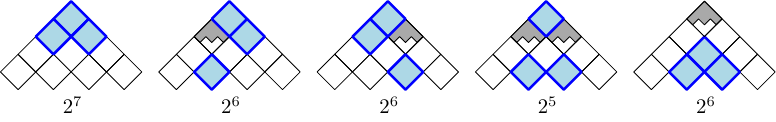}
\caption{Each Dyck path $\p$ of size $n$ with $m$ excited peaks
  (denoted in gray) yields $2^{2n-m+2}$ pleasant
  diagrams. For $n=3$, we have $C_{5}=5$ and $s_3=11$.  Thus,
  there are $|\ED(\delta_{3+2}/\delta_3)|=C_3 = 5$
  excited diagrams and $\sPD(\delta_{3+2}/\delta_3)=2^5\ts s_3=352$ pleasant
diagrams.} \label{fig:sch_pleasant}
\end{center}
\end{figure}

The proof is based on the following
corollary which is in turn a  direct application of Theorem~\ref{thm:num_pleasant}.
A {\em high peak} of a Dyck path $\p$ is a peak of height strictly
greater than one. We denote by $\HP(\p)$ the set of high peaks of~$\p$,
and by $\nonpeaks(\p)$ the points of the path that are not high
peaks. We use $2^{\mathcal{S}}$ denote the set of subsets of $\mathcal{S}$.

\begin{corollary} \label{lem:pleasant_is_subset_catalan}
The pleasant diagrams in $\PD(\delta_{n+2}/\delta_n)$ are in
bijection with
$$
\bigcup_{\p \in \Dyck(n)}
\Bigl(\ts\HP(\p) \times 2^{\nonpeaks(\p)}\ts\Bigr)\..$$
\end{corollary}

\begin{proof}
By Corollary~\ref{cor:excited2fandyckpaths} for the zigzag strip
$\delta_{n+2}/\delta_n$, $\NIP(\delta_{n+2}/\delta_n)$ is
the set of Dyck paths $\Dyck(n)$. Then by Theorem~\ref{thm:num_pleasant}, we have:
\[
\PD(\delta_{n+2}/\delta_n) = \bigcup_{\p \in \Dyck(n)}  \Bigl(\Lambda(\p) \times
2^{\p \setminus \Lambda(\p)}\Bigr).
\]
Lastly, note that the excited peaks of a Dyck path are exactly the
high peaks so $\Lambda(\p) = \HP(\p)$ and $\p \setminus \Lambda(p)=\nonpeaks(\p)$.
\end{proof}

\begin{proof}[Proof of Theorem~\ref{prop:pleas-schr}]
It is known (see~\cite{Deutsch}), that the number of Dyck paths of size~$n$
with~$k-1$ high peaks equals the {\em Narayana number} \ts
$N(n,k)\ts =\ts \frac{1}{n}\binom{n}{k}\binom{n}{k-1}$.
On the other hand, Schr\"{o}der numbers~$s_n$ can be written as
\begin{equation} \label{eq:Sch-Nar}
s_n \, = \, \sum_{k=1}^n \. N(n,k)\ts 2^{k-1}
\end{equation}
(see e.g.~\cite{Sulanke}).
By Lemma~\ref{lem:pleasant_is_subset_catalan}, we have:
\begin{equation} \label{eq:pfplesch}
\sPD(\delta_{n+2}/\delta_n) \, = \, \sum_{\p \in \Dyck(n)} \.
2^{|\nonpeaks(\p)|}\..
\end{equation}
Suppose Dyck path $\gamma$ has $k-1$ peaks, $1\leq k \leq n$.
Then $|\nonpeaks(\gamma)|= 2n+1-(k-1)$. Therefore, equation~\eqref{eq:pfplesch} becomes
$$
\sPD(\delta_{n+2}/\delta_n) \, = \, 2^{n+2}\. \sum_{k=1}^n \. N(n,k)\ts 2^{n-k}
\, = \, 2^{n+2}\. \sum_{k=1}^n \. N(n,n-k+1)\ts 2^{n-k} \, = \, 2^{n+2}\ts s_n\.,
$$
where we use the symmetry $N(n,k)=N(n,n-k+1)$ and~\eqref{eq:Sch-Nar}.
\end{proof}

In the same way as $|\ED(\delta_{n+2k}/\delta_n)|$ is given by a
determinant of Catalan numbers, preliminary computations suggest that
$\sPD(\delta_{n+2k}/\delta_n)$ is given by a determinant of
Schr\"oder numbers.

\begin{conjecture} \label{conj:plesantdetSch}
We have: \. $\sPD(\delta_{n+4}/\delta_n) \ts = \ts 2^{2n+5} ({s}_n \ts s_{n+2}\ts -\ts {s}_{n+1}^2)$.
More generally, for all $k \ge 1$, we have:
$$
\sPD(\delta_{n+2k}/\delta_n) \, = \, 2^{\binom{k}{2}} \det \bigl[\mathfrak{s}_{n-2+i+j}\bigr]_{i,j=1}^k\,,
\quad \text{where} \ \ \mathfrak{s}_n = 2^{n+2}s_n\..
$$
\end{conjecture}

Here we use $\mathfrak{s}_n=\sPD(\delta_{n+2}/\delta_n)$ in place of $s_n$ in the determinant to make the
formula more elegant.  In fact, the power of~$2$ can be factored
out.

\begin{remark}
This conjecture is somewhat unexpected since unlike the number of excited
diagrams, there is no known 
Lascoux--Pragacz-type identity for the number of pleasant diagrams (see Section~\ref{sec:LP4pleasant}).
\end{remark}



\medskip \subsection{$q$-analogue of Euler numbers via RPP}
We use our second $q$-analogue of the NHLF (Theorem~\ref{thm:skewRPP})
and Lemma~\ref{lem:pleasant_is_subset_catalan} to obtain identities
for the generating function of RPP of shape $\delta_{n+2}/\delta_n$ in
terms of Dyck paths. Recall the definition of $\EP_{n}(q)$ from the introduction:
\[
\EP_{n}(q) \, = \, \sum_{\sigma \in \Alt(n)} \. q^{\maj(\sigma^{-1} \kappa)}\.,
\]
where $\kappa=(13254\ldots)$. Note that $\maj(\sigma^{-1}\kappa)$ is the
sum of the descents of $\sigma \in \SS_n$ not involving both $2i+1$ and~$2i$.

\begin{example} \label{ex:qEulerB}
To complement Example~\ref{ex:qEulerA}, we have: \.
$\EP_1(q)=\EP_2(q) = 1$, \. $\EP_3(q)\ts =\ts q+1$, \.

\noindent
$\EP_4(q) \ts = \ts q^4+q^3+q^2+q+1$,
\ts and \. $\EP_5(q)\ts =\ts q^7+2q^6+2q^5+3q^4+3q^3+2q^2+2q+1$\ts.
\end{example}

\smallskip

\begin{proof}[Proof of Corollary~\ref{cor:euler-nhlf-rpp}]
By the theory of $P$-partitions,
the generating series of RPP of shape $\delta_{n+2}/\delta_n$ equals
\[
\sum_{\pi \in \RPP(\delta_{n+2}/\delta_n)} q^{|\pi|} \, = \,
\frac{\sum_{u \in \mathcal{L}(P_{\delta_{n+2}/\delta_n})} q^{\maj(u)}}{(1-q)(1-q^2)\cdots (1-q^{2n+1})}\,,
\]
where the sum in the numerator is over linear extensions $\mathcal{L}(P_{\delta_{n+2}/\delta_n})$ of the {\em
  zigzag} poset $P_{\delta_{n+2}/\delta_n}$ with the order preserving
labeling:
\begin{center}
\includegraphics[scale=0.8]{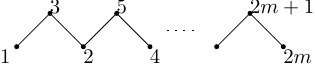}
\end{center}
These linear extensions
are in bijection with alternating permutations of size $2n+1$ and
\[
\EP_{2n+1}(q) \, = \, \sum_{\sigma\in \Alt_{2n+1}} \. q^{\maj(\sigma^{-1}\kappa)} \,
= \, \sum_{u\in \mathcal{L}(P_{\delta_{n+2}/\delta_n})} q^{\maj(u)}\..
\]
Thus
\begin{equation} \label{eq:RPP2ndqEuler}
\sum_{\pi \in \RPP(\delta_{n+2}/\delta_n)} q^{|\pi|} \, = \,
\frac{\EP_{2n+1}(q)}{(1-q)(1-q^2)\cdots (1-q^{2n+1})}\,.
\end{equation}
By Theorem~\ref{thm:num_pleasant} (see also \cite[\SS 6.4]{MPP1}) for the skew shape
$\delta_{n+2}/\delta_n$ and \eqref{eq:RPP2ndqEuler}, we have:
\begin{equation} \label{eq:pf2ndqEuler}
\sum_{D\in
  \ED(\delta_{n+2}/\delta_n)} q^{a'(D)} \prod_{u\in [\lambda]\setminus D}\frac{1}{1-q^{h(u)}} \,
=  \, \frac{\EP_{2n+1}(q)}{(1-q)(1-q^2)\cdots (1-q^{2n+1})}\,,
\end{equation}
where $a'(D)=\sum_{u \in \Lambda(D)} h(u)$. By the proof of
Lemma~\ref{lem:pleasant_is_subset_catalan}, if $D \in \ED(\delta_{n+2}/\delta_n)$
corresponds to the Dyck path~$\p$ then excited peaks $u\in \Lambda(D)$
correspond to high peaks $(c,d) \in \HP(\p)$ and $h(u)=2d+1$.  Using this
correspondence, the LHS of~\eqref{eq:pf2ndqEuler} becomes the LHS of the desired expression.
\end{proof}

Finally, preliminary computations suggest the following analogue
of Corollary~\ref{cor:qEdet}.
\begin{conjecture} \label{conj:RPPEdet}
We have:
$$
\sum_{\pi \in \RPP(\delta_{n+2k}/\delta_n)} q^{|\pi|} =
q^{-N}  \det
\left[\wt{E}^\ast_{2(n+i+j)-3}(q)\right]^{k}_{i,j=1}\,,
$$
where \ts $N  = k(k-1)(6n+8k-1)/6$ \ts and \ts
$\wt{E}^\ast_{k}(q) =\EP_{k}(q)/(1-q)\cdots (1-q^{k})\ts.$
\end{conjecture}

\bigskip

\bigskip\section{Final remarks} \label{sec:finrem}

\subsection{} \label{sec:compare}

Other known formulas for $f^{\lambda/\mu}$ are the {\em Jacobi--Trudi
  identity}, the {\em Littlewood--Richardson rule}, and the
{\em Okounkov--Olshanski formula} \cite{OO}. We discuss these and other less-known
formulas for $f^{\lambda/\mu}$ coming from \emph{equivariant Schubert
  structure constants} in \cite[\S 9]{MPP1}.

The Jacobi-Trudi identity is one of the first nontrivial formulas to
count $f^{\lambda/\mu}$. In this paper we have unveiled a strong relation between the
Lascoux--Pragacz identity for Schur functions and the NHLF for
$f^{\lambda/\mu}$. As mentioned in Section~\ref{subsec:LP}, Hamel and Goulden
\cite{HGoul} unified these two identities
into an exponential family of determinantal identities of Schur
functions. Chen--Yan--Yang
\cite{CYY} gave a method to transform among these identities. It would
be of interest if other formulas for $f^{\lambda/\mu}$, like the ones
mentioned above, are related
to special cases of Hamel--Goulden identities.

\subsection{} \label{ss:comparisonKonvalinka}

In~\cite{Ko}, Konvalinka gives a new proof of the NHLF.  Specifically,
he presents a bumping algorithm on bicolored flagged
tableaux of shape $\mu$ to prove the Pieri--Chevalley formula for
general skew shapes (see \cite[\S 8.4]{IkNa09} and \cite{MPP3}). For
the border strips this approach is different from our proof of
Lemma~\ref{lem:ChevalleyStrips}.  In fact, our proof uses the underlying
single path in the excited diagrams of a border strip to perform
cancelations. While Konvanlinka's proof is subtraction-free,
it involves an insertion on the inner partition~$\mu$ that
could be arbitrary even for a border strip.

It is worth noting that both proofs of Lemma~\ref{thm:NHLFborderstrips}
are quite technical.  Initially, this came as a surprise to us, and
our effort to understand the underlying multivariate algebraic
identities led to~\cite{MPP3}.
Let us also mention that in~\cite{IkNa09}, the authors use the
Pieri--Chevalley formula for Billey's \cite{Bil} {\em Kostant polynomials} to prove the
analogue of Lemma~\ref{lem:ChevalleyStrips} for all skew shapes.

\subsection{} \label{ss:finrem-hist-survey}
There is a very large literature on alternating permutations,
Euler numbers, Dyck paths, Catalan and Schr\"{o}der numbers, which are some
of the classical combinatorial objects and sequences.  We refer to~\cite{Stanley_SurveyAP}
for the survey on the first two, to~\cite{StCat} for a thorough treatment of
the last three, and to~\cite{GJ,OEIS,EC2} for various generalizations, background
and further references.

Finally, the first $q$-analogue $E_n(q)$ of Euler numbers we consider is
standard in the literature and satisfies a number of natural properties,
including a $q$-version of equation~\eqref{eq:tan-sec}
(see e.g.~\cite[$\S$4.2]{GJ}).  However, the second $q$-analogue
$\EP_n(q)$ appears to be new.  It would be interesting
to see how it fits with the existing literature of multivariate
Euler polynomials and statistics on alternating permutations.

\subsection{}  \label{ss:finrem-cat}
The curious Catalan determinant in Corollary~\ref{cor:cat-det}
appeared in the first {\tt arXiv} version of \cite{MPP1}. However, the
proof here is more self contained as a direct application of
Theorem~\ref{thm:num_excited_HG}. This Catalan determinant is both
similar and related\footnote{The connection was found by T.~Amdeberhan (personal communication).}
to another Catalan determinant
in~\cite[proof of Lemma~1.1]{AL}.
In fact, both determinants are special cases of more general counting results,
and both can be proved by the Lindstr\"om--Gessel--Viennot lemma.

%

\subsection{} \label{ss:finrem-foulkes}
The connection between alternating permutations and symmetric functions
of border strips goes back to Foulkes~\cite{Fou}, and has
been repeatedly generalized and explored ever since
(see~\cite{Stanley_SurveyAP}).  It is perhaps surprising that
Corollary~\ref{cor:euler-nhlf} is so simple, since the other two
positive formulas in Section~\ref{sec:compare} become quite involved.
For the LR-coefficients, let partition $\nu\vdash 2n+1$ be such that
$\ts \nu_1, \ell(\nu) \le n+1$.  It is easy to see that in this case
the corresponding LR-coefficient is nonzero:
\ts $c_{\de_n\ts\nu}^{\de_{n+2}} >0$, suggesting that
summation over all such~$\nu$ would can be hard to compute.

%
\subsection{}  \label{ss:finrem-RPP}
%
%
The proof in~\cite{MPP1} of the skew RPP $q$-analogue of
Naruse (Theorem~\ref{thm:skewRPP}) is already bijective using the
Hillman--Grassl correspondence. It would be interesting to see if for RPP the case
for border strips implies the case for all connected skew shapes.
Note that we do not know of a Lascoux--Pragacz analogue of~\eqref{eq:LascouxPragacz-SSYT}
for skew RPP. Relatedly, we also do not know of such an analogue for the number of pleasant diagrams (the supports of arrays obtained from
Hillman--Grassl applied to skew RPP) as discussed in Section~\ref{sec:LP4pleasant}.
On the other hand, Conjectures~\ref{conj:RPPEdet} and \ref{conj:plesantdetSch}  
suggest that there might be such formulas in some cases.


\vskip.86cm

\subsection*{Acknowledgements}
We are grateful to Per Alexandersson, Dan Betea, Sara Billey, Jang Soo
Kim, Matja\v{z} Konvalinka, Leo Petrov, Robert Proctor, Eric Rains,
Luis Serrano, Richard Stanley, Matthew Willis, and Alex Yong
for useful comments and help with the references.  We are also thankful to
Tewodros Amdeberhan, Brendon Rhoades and Emily Leven for discussions on the
curious Catalan determinants, and to the referees for helpful comments
and suggestions.  The first author is supported by an AMS-Simons
travel grant. The second and third authors were partially supported by the~NSF.

\vskip.8cm



\begin{thebibliography}{abcdefg}

\bibitem[AR]{HEC}
R.~Adin and Y.~Roichman, Standard {Y}oung tableaux, in
{\em Handbook of Enumerative Combinatorics} (M.~B\'ona, editor),
CRC Press, Boca Raton, 2015, 895--974.

\bibitem[AL]{AL}
T.~Amdeberhan and E.~Leven,
Multi-cores, posets, and lattice paths, \emph{Adv.\ Appl.\ Math.}~\textbf{71} (2015), 1--13.


\bibitem[AJS]{AJS}
H.H.~Andersen, J.C.~Jantzen and W.~Soergel,
Representations of quantum groups at $p$-th root of unity and of
semisimple groups in characteristic $p$: independence of $p$,
\emph{Ast\'{e}risque}, 220, 321.



\bibitem[BR]{BR}
Y.~Baryshnikov and D.~Romik,
Enumeration formulas for {Y}oung tableaux in a diagonal strip,
{\em Israel J.~Math.} \textbf{178} (2010), 157--186.

\bibitem[Bil]{Bil}
S.~Billey,
 {K}ostant polynomials and the cohomology ring for $G/B$,
 {\em Duke Math.~J.} \textbf{96} (1999), 205--224.

\bibitem[BGR]{BGR}
A.~Borodin, V.~Gorin and E.~M.~Rains,
$q$-distributions on boxed plane partitions,
\emph{Selecta Math.}~\textbf{16} (2010), 731--789.


\bibitem[CYY]{CYY}
W.~Y.~C.~Chen, A.~L.~B.~Yan and G.-G.~Yang,
Transformations of border strips and {S}chur function determinants,
{\em J.~Algebraic Combin.}~\textbf{21} (2005),  379--394.




\bibitem[SV]{StCath-Viennot}
M.~de~Sainte-Catherine and X.~G.~Viennot,
\newblock Enumeration of certain {Y}oung tableaux with bounded height,
in {\em Lecture Notes Math.}~\textbf{1234}, Springer,
Berlin, 1986, 58--67.

\bibitem[Deu]{Deutsch}
E.~Deutsch, An involution on {D}yck paths and its consequences,
{\em Discrete Math.}~\textbf{204} (1999), 163--166.

\bibitem[Eli]{Elizalde}
S.~Elizalde,
A bijection between $2$-triangulations and pairs of non-crossing {D}yck paths,
{\em J.~Combin.\ Theory, Ser.~A}~\textbf{114} (2007), 1481--1503.


\bibitem[Fis]{Fis}
I.~Fischer, A bijective proof of the hook-length formula for shifted standard tableaux,
{\tt arXiv:math/0112261}.



\bibitem[FK]{FK}
S.~Fomin and A.N.~Kirillov,  Reduced words and plane partitions,
\emph{J. of Algebraic Combin.}~\textbf{6} (1997), 311--319.

\bibitem[Fou]{Fou}
H.~O.~Foulkes,
Tangent and secant numbers and representations of symmetric groups,
\emph{Discrete Math.}~\textbf{15} (1976), 311--324.

\bibitem[FRT]{FRT}
J.~S. Frame, G.~de~B. Robinson and R.~M. Thrall,
The hook graphs of the symmetric group,
{\em Canad.\ J.\ Math.}~\textbf{6} (1954), 316--324.






\bibitem[GV]{GV}
I.~M.~Gessel and X.~G.~Viennot,
Determinants, paths and plane partitions, preprint (1989);
available at \. \href{http://tinyurl.com/zv3wvyh}{tinyurl.com/zv3wvyh}.

\bibitem[GJ]{GJ}
I.~P.~Goulden and D.~M.~Jackson,
\emph{Combinatorial enumeration}, John Wiley, New York, 1983.


\bibitem[GK]{GK}
W.~Graham and V.~Kreiman,
Excited Young diagrams, equivariant K-theory, and Schubert varieties,
\emph{Trans.\ AMS}~\textbf{367} (2015), 6597--6645.

\bibitem[GNW]{GNW}
C.~Greene, A.~Nijenhuis and H.~S.~Wilf,
A probabilistic proof of a formula for the number of {Y}oung tableaux of a given shape,
{\em Adv.\ Math.}~\textbf{31} (1979), 104--109.


\bibitem[HW]{HW}
Z.~Hamaker and N.~Williams,
Subwords and plane partitions,
{\em Proceedings FPSAC 2015}, 241--252.

\bibitem[HPPW]{HPPW}
Z.~Hamaker, R.~Patrias, O.~Pechenik, and N.~Williams,
Doppelg\"{a}ngers: bijections of plane partitions,  {\tt arXiv:1602.05535}.

\bibitem[HaG]{HGoul}
A.~M.~Hamel and I.~P.~Goulden,
Planar decompositions of tableaux and {S}chur function determinants,
{\em European J.~Combin.}~\textbf{16} (1995), 461--477.



\bibitem[HiG]{HG}
A.~P.~Hillman and R.~M.~Grassl,
Reverse plane partitions and tableau hook numbers,
{\em J.~Combin.\ Theory, Ser.~A}~\textbf{21} (1976), 216--221.


\bibitem[IN]{IkNa09}
 T.~Ikeda and H.~Naruse,
 Excited {Y}oung diagrams and equivariant {S}chubert calculus,
 {\em Trans.\ AMS}~\textbf{361} (2009), 5193--5221.
%
%

\bibitem[Jin]{Jin}
E.Y.~Jin,
Outside nested decompositions of skew diagrams of skew diagrams and
Schur functions determinants, {\tt arXiv:1606.01764}


\bibitem[Jon]{Jonsson}
J.~Jonsson,
Generalized triangulations and diagonal-free subsets of stack
  polyominoes,
{\em J.~Combin.\ Theory, Ser.~A}~\textbf{112} (2005), 117--142.









\bibitem[K1]{Kratt-invol}
C.~Krattenthaler,
Bijective proofs of the hook formulas for the number of standard Young tableaux,
ordinary and shifted, \emph{Electron.\ J.~Combin.}~\textbf{2} (1995), RP~13, 9~pp.




\bibitem[Kon]{Ko}
M.~Konvalinka, A bijective proof of the hook-length formula for skew
shapes, preprint 2016.

\bibitem[Kre1]{VK}
V.~Kreiman,
{S}chubert classes in the equivariant {K}-theory and equivariant
  cohomology of the {G}rassmannian, {\tt arXiv:math.AG/0512204}.

\bibitem[Kre2]{VK2}
V.~Kreiman, {S}chubert classes in the equivariant {K}-theory and equivariant
  cohomology of the {L}agrangian {G}rassmannian, {\tt arXiv:math.AG/0602245}.


\bibitem[LasP]{LP}
A.~Lascoux and P.~Pragacz, Ribbon {S}chur functions,
{\em European J. Combin.}~\textbf{9} (1988), 561--574.



\bibitem[MPP1]{MPP1}
A.H.~Morales, I.~Pak and G.~Panova, Hook formulas for skew shapes~I. $q$-analogues and bijections,
{\tt arXiv:1512.08348}.

\bibitem[MPP2]{MPP3}
A.H.~Morales, I.~Pak and G.~Panova,
Hook formulas for skew shapes~III. Multivariate and product formulas, in preparation.

\bibitem[MPP3]{MPP4}
A.H.~Morales, I.~Pak and G.~Panova,
Hook formulas for skew shapes~IV. Increasing tableaux and factorial Grothendieck polynomials,
in preparation.

 \bibitem[MS]{MS}
 A.~I. Molev and B.~E. Sagan,
 A {L}ittlewood-{R}ichardson rule for factorial {S}chur functions,
 {\em Trans.~AMS}~\textbf{351} (1999), 4429--4443.

\bibitem[Naru]{Strobl}
H.~Naruse, {S}chubert calculus and hook formula, talk slides at \emph{73rd
{S}\'em.\ {L}othar.\ {C}ombin.},
Strobl, Austria, 2014; available at \.
  \href{http://www.emis.de/journals/SLC/wpapers/s73vortrag/naruse.pdf}{tinyurl.com/z6paqzu}.



\bibitem[NPS]{NPS}
J.-C.~Novelli, I.~Pak and A.~V.~Stoyanovskii,
A direct bijective proof of the hook-length formula,
{\em Discrete Math.\ Theor.\ Comput.\ Sci.}~\textbf{1} (1997), 53--67.

\bibitem[OO]{OO}
A.~Okounkov and G.~Olshanski,
\newblock Shifted {S}chur functions,
{\em St.~Petersburg Math.~J.}~\textbf{9} (1998), 239--300.


\bibitem[Pak]{P1}
I.~Pak,
Hook length formula and geometric combinatorics.
{\em S\'em.\ Lothar.\ Combin.}~\textbf{46} (2001), Art.~B46f, 13~pp.



\bibitem[Pr1]{Pr1}
R.~A.~Proctor,
Shifted plane partitions of trapezoidal shape.
{\em Proc.\ AMS}~\textbf{89} (1983), 553--559.

\bibitem[Pr2]{Pr2}
R.~A.~Proctor, New symmetric plane partition identities from invariant theory work of De Concini
and Procesi, \emph{Europ.\ J.\ Combin.}~\textbf{11} (1990), 289--300.



\bibitem[Rem]{Rem}
J.~B.~Remmel,
Bijective proofs of formulae for the number of standard Young tableaux,
\emph{Linear and Multilinear Algebra}~\textbf{11} (1982), 45--100.

\bibitem[Sag1]{Sag-hook}
B.~E. Sagan, Enumeration of partitions with hooklengths,
\emph{European J.~Combin.}~\textbf{3} (1982), 85--94.

\bibitem[Sag2]{SSG}
B.~E. Sagan, {\em The Symmetric Group}, Springer, 2000.

\bibitem[Sag3]{BStree}
B.~E. Sagan, Probabilistic proofs of hook length formulas involving trees,
{\em S\'em.\ Lothar.\ Combin.}~\textbf{61A} (2009), Art.~B61Ab, 10~pp.

\bibitem[Sage]{sage-combinat}
The {S}age-{C}ombinat community.
\newblock {S}age-{C}ombinat: enhancing {S}age as a toolbox for computer
  exploration in algebraic combinatorics, 2013.
\newblock \url{http://combinat.sagemath.org}.

 \bibitem[SS]{SS}
 L.~Serrano and C.~Stump,
 Maximal fillings of moon polyominoes, simplicial complexes, and
  {S}chubert polynomials, {\em Electron.\ J.~Combin.}~\textbf{19} (2012),
  RP~16, 18 pp.

\bibitem[OEIS]{OEIS}
N.~J.~A.~Sloane,
The {O}nline {E}ncyclopedia of {I}nteger {S}equences,
\href{http://oeis.org}{\tt oeis.org}.

\bibitem[S1]{St71}
R.~P.~Stanley, Theory and applications of plane partitions, Part~2,
{\em Stud.\ Appl.\ Math.}~\textbf{50} (1971), 259--279.


\bibitem[S2]{StIC}
R.~P.~Stanley, Plane partitions, past, present, and future,
Proceedings Third International Conference (New York 1985), 397--401.




\bibitem[S3]{Stanley_SurveyAP}
R.~P.~Stanley, A survey of alternating permutations, in
{\em Combinatorics and graphs}, AMS, Providence, RI, 2010, 165--196.

\bibitem[S4]{EC2}
R.~P.~Stanley, {\em Enumerative Combinatorics}, vol.~1 (second ed.)
and vol.~2 (first ed.), Cambridge Univ.~Press, 2012 and~1999.


\bibitem[S5]{StCat}
R.~P.~Stanley, {\em {C}atalan Numbers},
Cambridge Univ.~Press, 2015.


\bibitem[Sul]{Sulanke}
R.~A.~Sulanke, The {N}arayana distribution,
{\em J.~Statist.\ Plann.\ Inference}~\textbf{101} (2002), 311--326.







\bibitem[Wac]{W85}
M.~Wachs, Flagged {S}chur functions, {S}chubert polynomials, and symmetrizing
  operators,  {\em J.~Combin. Theory, Ser.~A}~\textbf{40} (1985), 276--289.




\bibitem[Zei]{Zei}
D.~Zeilberger,
A short hook-lengths bijection inspired by the Greene--Nijenhuis--Wilf proof,
\emph{Discrete Math.}~\textbf{51} (1984), 101--108.

\end{thebibliography}
\end{document}